\documentclass[reqno,11pt]{amsart}
\usepackage{amssymb, amsmath,amsmath,latexsym,amssymb,amsfonts,amsbsy, amsthm,mathtools,graphicx,CJKutf8,CJKnumb,CJKulem,color,cases,empheq,extpfeil,bm}
\renewcommand{\theequation}{\thesection.\arabic{equation}}
\numberwithin{equation}{section}

\allowdisplaybreaks
\newcommand{\cal }{\mathcal }

\newcommand{\dif}{\,\mathrm{d}}
\newcommand\supp{\mathrm{supp}}

\newtheorem{remark}{Remark}[section]
\newtheorem{lemma}{Lemma}[section]

\newtheorem{prop}{Proposition}[section]

\newtheorem*{main}{Main theorem}

\begin{document}

%%%%%%%%%%%%%%%%%%%%%%%%%%%%%%%%%%%%%%%%%%%%%%%%%%%%%%%%%%%%%%%%%%
%% %%%%                                                                            Introduction                                                                             %%%%%%

\title[Global Well-Posedness of 2D Euler-$\alpha$ Equations in Exterior Domain]
{Global Well-Posedness of 2D Euler-$\alpha$ Equation in Exterior Domain}
\author[You]{Xiaoguang You}
\address{Xiaoguang You $\newline$ School of Mathematics, Northwest University, Xi'an 710069, P. R. China}
\email{wiliam$\_$you@aliyun.com}

\author[Zang]{Aibin Zang}
\address{Aibin Zang $\newline$ School of Mathematics and Computer Science \& The Center of Applied Mathematics, Yichun university, Yichun, Jiangxi, 340000 P. R. China}
\email{zangab05@126.com}

%%%%P13 address changed (we have some new administrative rules)

%\date{}%
%\dedicatory{}%
%\commby{}%
% --------------------------------------------------------------
\begin{abstract}
After casting Euler-$\alpha$ equations into vorticity-stream function formula, we obtain some very useful estimates from the properties of the vorticity formula in exterior domain. Basing on these estimates, one can have got the global existence and uniqueness of the solutions to Euler-$\alpha$ equations in 2D exterior domain provided that the initial data is regular enough.
\end{abstract}

%\date{July 21, 2015}

\maketitle
% -----------------------------------------

\noindent {\sl Keywords\/}:  Euler-$\alpha$ equations; Exterior Domain;  Dirichlet boundary conditions; vorticity-stream function
\vskip 0.2cm

\noindent {\sl AMS Subject Classification} (2020): 35Q53; 35B30; 35G25 \\

%%%%%%%%%%%%%%%%%%%%%%%%%%%%%%%%%%%%%%%%%%%%%%%%%%%%%%%%%%%%%%
%%%%%%%%%%%%%%%%%%%%%%%%%%%%%%%%%%%%%%%%%%%%%%%%%%%%%%%%%%%%%
\renewcommand{\theequation}{\thesection.\arabic{equation}}
\setcounter{equation}{0}

%%%%%%%%%%%%%%%%%%%%%%%%%%%%%%%%%%%%%%%%%%%%%%%%%%%%%%%%%%%%%%
%%%%%%%%%%%%%%%%%%%%%%%%%%%%%%%%%%%%%%%%%%%%%%%%%%%%%%%%%%%%%%%%%%
\section{Introduction}
Let $\mathcal{O} \subset \mathbb{R}^2$ be a bounded, simply connected domain with $C^\infty$ Jordan boundary $\Gamma$. We consider Euler-$\alpha$ equations in exterior domain $\Omega=\mathbb{R}^2 \setminus \overline{\mathcal{O}}$
%, with initial data $\bm{u}_0^{\alpha} \in \bm{H}^s(\Omega) \cap V$,
given by:

\begin{numcases}{}
    \partial_t \bm{v} + \bm{u} \cdot {\nabla} \bm{v}+ \sum_{j=1}^2 v_j\nabla u_j + \nabla p = 0
    \qquad    \qquad  &  $\text{in}  \ \Omega \times (0, \infty), $ \label{euler-alpha-1}  \qquad \      \qquad \\
    \text{div} \ \bm{u}= 0 &  $\text{in} \ \Omega \times [0, \infty) $ \label{euler-alpha-2}
    %\bm{u}^{\alpha} = 0  & $\text{in} \ \Gamma \times [0, \infty] \label{euler-alpha-3}$,\\
    %\bm{u}^{\alpha}|_{t=0} = \bm{u}_0^{\alpha} & $\text{in} \ \Omega $\label{euler-alpha-4},\\
    %\bm{u}^{\alpha} (x, t) \to 0 & $\forall t \in [0, \infty), |x| \to \infty $ \label{euler-alpha-5}
\end{numcases}
with initial data $\bm{u}_0$, where %$s \geq 3$ and
$\bm{v}= \bm{u}- \alpha^2 \Delta \bm{u}$ and $p$ is the pressure. Above $\bm{u}$ is called the \textit{filtered} velocity, while $
\bm{v}$ is the \textit{unfiltered} velocity and $\nabla^ \perp \cdot \bm{v}$(noted by $q$) is the \textit{unfiltered} vorticity.

Recall that $\alpha$ is a parameter having the dimensions of a length, and when Euler-$\alpha$ \eqref{euler-alpha-1} and \eqref{euler-alpha-2} is used in Large Eddy Simulations, this parameter should be related to the smallest resolved scale, we can refer \cite{2006Mathematics}. Euler-$\alpha$ equations are viewed as a conservative generalization of incompressible Euler equations. Generally, Euler-$\alpha$ equations go back to at least three independent developments \cite{kouranbaeva2000global}.

First, Euler-$\alpha$ equations arise as a generalization into several spatial dimensions of the scalar Fobas-Fuchssteiner-Camassa-Holm (FFCH) equation
\begin{equation}\label{FFCH}
\nonumber u_t-u_{xxt}+2\kappa u_x=-3uu_x+2u_xu_{xx}+uu_{xxx},
\end{equation}
which models unidirectional surface wave in shallow water. The FFCH equation can be realized as a geodesic equation on the diffeomorphism group of the circle equipped with a metric that arises from $H^1$ metric on its Lie algebra of vector field. It is this property that can be used to obtain Euler-$\alpha$ equations, an idea introduced by Holm et.al \cite{holm1998euler,holm1998euler2}.

Second, Euler-$\alpha$ equations are the equations of motion of an inviscid non-Newtonian fluid of second grade. In this interpretation, $\alpha^2$ is a material constant expressing the elastic response of the fluid. Due to essential material frame indifference and observer objectivity on Rivlin-Ericksen law, Markovitz and Coleman \cite{markovitz1964incompressible} and Noll and Truesdell \cite{truesdell2004non} obtain Euler-$\alpha$ equations without viscosity in second grade fluid.

Third, Euler-$\alpha$ model can be seen as the continuous analog of a computational vortex blob method for two-dimensional hydrodynamics, we can refer the review by Leonard \cite{leonard1980vortex}.

There has been substantial work on the well-posedness for Euler-$\alpha$ equations. Busuioc \cite{oliver2001vortex} proved the global existence and uniqueness of the solutions for  Euler-$\alpha$ equations in $\mathbb{R}^2$ and local well-posedness in $\mathbb{R}^3$. The second author has obtained  the global existence for two-dimensional period domain and the local existence in some uniform time with respect to the parameter $\alpha$  for three-dimensional period domain in \cite{zang}. In bounded domains with Dirichlet boundary conditions, namely
\begin{equation}\label{Dirichlet}
\bm{u}=0  \qquad \text{on } \Gamma,
\end{equation}
Shkoller \cite{shkoller2000analysis} showed that Euler-$\alpha$ equations are well-posed by transferring the problem from Eulerian to the Lagrangian setting. Moreover,  the solution is global for 2D case. In Eulerian setting, well-posedness for 2D case was proved in \cite{lopes2015convergence}, by Banach fixed point theorem. In half plane, the global existence of weak solutions were estabished in \cite{busuioc2020limit} in the space of Radon measures.

It is an important issue to investigate the well-posedness of Euler-$\alpha$ flow around an obstacle. This work is inspired by the singular limits for Euler-$\alpha$ equations to Euler equations as $\alpha\to 0$ or both $\alpha$ and the radius of obstacle go to zero. The latter is analogy to the results of \cite{iftimie2009incompressible} for  the viscous incompressible flow around a small obstacle. To arrive these limits, we must firstly discuss on the well-posedness problem of Euler-$\alpha$ equations in the exterior domain.

Therefore, we focus on investigating the global existence and uniqueness of the system \eqref{euler-alpha-1}, \eqref{euler-alpha-2}  with initial velocity $\bm{u}_0$, the boundary conditions \eqref{Dirichlet} and $\bm{u}$ decays to zero at infinity, that is
\begin{equation}\label{inftydecay}
 \bm{u}\to 0~~ \mbox{as}~~ |x|\to +\infty.
 \end{equation}
To solve the exterior problem, one can meet with two main difficulties: one is from the nonlinearities with high order derivatives; other is from the unboundedness of the domain.

To overcome these difficulties, we will cast the system \eqref{euler-alpha-1}, \eqref{euler-alpha-2} into the following vorticity-stream fuction formula
\begin{numcases}{}
      \partial_t q + \bm{u} \cdot \nabla{q} = 0 &  $\text{in} \ \Omega \times [0, \infty),$ \label{vstream1}\\
      \Delta \psi  = q &  $\text{in} \ \Omega \times [0, \infty) $, \label{vstream2}\\
     {\bm{u}} - \alpha^2 \Delta {\bm{u}} + \nabla p = \nabla^\perp \psi & $\text{in} \ \Omega \times [0, \infty) $\label{vstream3}
\end{numcases}
here $\psi$ is called the stream function of the \textit{unfiltered} velocity ${\bm{v}}(\equiv {\bm{u}} - \alpha^2 \Delta {\bm{u}})$ with the divergence-free vector field $\bm{u}$. It is easy to see that the system \eqref{vstream1},\eqref{vstream2} and \eqref{vstream3} is an elliptic-hyperbolic coupled system. To end it, we must linearize the system \eqref{vstream1},\eqref{vstream2} and \eqref{vstream3}. By solving the Poisson equations \eqref{vstream2} for the stream function $\psi$ with $q$, we can obtain the uniform bound of high order derivatives of $\psi.$ Applying the Galdi's estimate \cite{galdi2011introduction} to the Stoke type equations \eqref{vstream3}, we are thus able to find the uniform estimates of high order derivatives for the velocity $\bm{u}$.

However, since the domain is unbounded, it is nontrivial to establish a prior estimates of low order derivatives of $\bm{u}$ from ($\ref{vstream1}$)--($\ref{vstream3}$). To end it, we assume temporarily that the initial \textit{unfiltered} vorticity $q_0=\nabla^\perp \cdot (\bm{u}_0-\alpha^2\Delta \bm{u}_0) $ has a compact support. Observing that $q$ satisfies  a transport equation, one shows that $q$, in finite time interval, still has a compact support. Then by the estimates of the Poisson equation, we can get the uniform bound for the stream function $\psi$. Therefore, by Banach fixed point theorem, we can find a unique triple $(\bm{u},\psi,q)$ satisfies the system \eqref{vstream1}, \eqref{vstream2} and \eqref{vstream3} provided that $q_0$ has a compact support.

Finally, using  the previous  uniform estimates of ${\bm{u}}$ and $\psi$ which are independent of the support of $q$ and the system \eqref{euler-alpha-1},\eqref{euler-alpha-2},  we can immediately obtain the uniform estimates of lower order derivatives of $\bm{u}$, which help to eliminate the assumption for compact support of $q_0$.

The article is organized as follows. In section 2, we will introduce some notations and preliminaries. In section 3, we will present some technical lemmas and their proofs. In section 4, we will state and prove main result. In last section,  we will give some comments and discussions.

%%%%%%%%%%%%%%%%%%%%%%%%%%%%%%%%%%%%%%%%%%%%%%%%%%%%%%%%%%%%%%
%%%%%%%%%%%%%%%%%%%%%%%%%%%%%%%%%%%%%%%%%%%%%%%%%%%%%%%%%%%%%
\renewcommand{\theequation}{\thesection.\arabic{equation}}
\setcounter{equation}{0}

%%%%%%%%%%%%%%%%%%%%%%%%%%%%%%%%%%%%%%%%%%%%%%%%%%%%%%%%%%%%%%
%%%%%%%%%%%%%%%%%%%%%%%%%%%%%%%%%%%%%%%%%%%%%%%%%%%%%%%%%%%%%%%%%%
%%%%%%%%%%%%%%%%%%%%%%%%%%%%%%%%%%%%%%%%%%%%%%%%%%%%%%%%%%%%%%%%%%%%%%%%%%%%%%%%%%%%%%%%%%%%%%%%%%%%%%%%%%%%%%%%%%%%%%%%%%%%%%%%%%%%%%%%%%%%%%%%%%%%%%%%%%%%%%%%%%%%%%
\section{Notations and preliminaries} \label{notation}
In this section, we introduce notations and present preliminary results. We use the notation $H^s(\Omega)$ for the usual $L^2$-based Sobolev spaces of order $s$. $C_0^\infty(\Omega)$ represents the space of smooth functions with infinitely many derivatives, compactly supported in $\Omega$, and $H_0^s(\Omega)$ the closure of $C_0^\infty$ under the $H^s$-norm. For the sake of simplicity, $\bm{H}^s(\Omega)$(respectively $\bm{H}_0^s(\Omega)$) stands for vector space $(H^s(\Omega))^2$(respectively $(H^s_0(\Omega))^2$).  Let $L > 0$ be arbitrary, we set $\Omega_{L} := \Omega \cap B(0, L)$ and $\Omega^L := \Omega \setminus \overline{B(0, L)}$, where $B(0, L)$ is the disk centered at origin with radius $L$. By the way, the unit disk centered at origin in $\mathbb{R}^2$ is denoted by $D$. Let $\mathcal{A}$  be an arbitrary set of $\mathbb{R}^2$, $\delta(\mathcal{A})$ represents the $\textit{diameter}$ of $\mathcal{A}$, that is
\begin{align}
\nonumber \delta(\mathcal{A}) := \sup_{x, y \in \mathcal{A}} |x - y|
\end{align}
We also make use of the following notations
\begin{align*}\label{2}
  %& H := \lbrace \textbf{u} \in (L^2(\Omega))^2:  \text{div} \ \textbf{u} = 0 \text{ in } \Omega,  \textbf{u} \cdot \nu = 0 \rbrace\\
  & \mathcal{D} := \lbrace \textbf{u} \in (C_0^\infty(\Omega))^2; \text{div}\, \textbf{u} = 0\rbrace,\\
  & V  := \lbrace \textbf{u} \in \bm{H}_0^1(\Omega); \ \text{div} \, \textbf{u}  =  0  \ \text{in} \ \Omega\rbrace,\\
  & L^2_{\sigma} := \lbrace \textbf{u} \in (L^2(\Omega))^2; \ \text{div}\, \textbf{u} = 0, \textbf{u} \cdot \nu|_{\Gamma} = 0 \rbrace, \\
  & X^2_{har} := \lbrace \textbf{h} \in (L^2({\Omega}))^2; \text{div}\, \textbf{h} = 0, \text{rot}\, \textbf{h} = 0, \textbf{h} \cdot \nu|_{\Gamma} = 0\rbrace,\\
& \bm{\dot{H}} := \{ \bm{u} \in (L_{loc}^2(\Omega))^2; \int_\Omega|\nabla \bm{u}|^2\dif x<\infty\},
  %& \mathcal{M} := \{\bm{u} \in C(\bar{\Omega}); \nabla\bm{u} \in \bm{H}^1(\Omega)\}
\end{align*}
 here $\nu$ is the nomral vector to $\Gamma$.

We then present  some well-known results which will be used in this paper. The following lemma, in the matter of transport equation, can be found  in \cite{diperna1989ordinary}.
\begin{lemma}\label{lemma-transport}{}
Let $p \in [1, \infty), q_0 \in L^p(\Omega), \textbf{u}\in L^\infty( [0, T]; V)$ with $T > 0$ fixed, then the transport equation:
\begin{numcases}{}
    \partial_t q + \textbf{u} \cdot \nabla{q} = 0, \label{transport-1}\\
    q|_{t=0} = q_0 \label{transport-2}
\end{numcases}
has a unique weak solution  q $\in C([0, T]; L^p(\Omega))$, and satisfying:
\begin{equation}\label{transport-3}
     \begin{split}
        \sup_{t\in[0, T]}\|q(t)\|_{L^p({\Omega})} \leqslant \|q_0\|_{L^p({\Omega})}.
     \end{split}
\end{equation}
\end{lemma}
\begin{remark}\label{remark-transport-equation} Assume that $\delta(\supp \,q_0) < \infty$ and $\bm{u} \in L^1([0, T]; V \cap \bm{H}^2(\Omega))$, then we have
\begin{equation} \label{remark-equation-1}
  \delta(\supp \,q(t)) < \delta(\supp \,q_0) + C\int_0^t\|\bm{u}(s)\|_{\bm{H}^2(\Omega)} \dif s,
\end{equation}
for $t \in [0, T]$, where $C$ is a constant depends only on $\Omega$.

Indeed, let $\bm{X}\in C([0, T]; (C(\Omega))^2)$ be the unique weak solution of the following ordinary differential equations
\begin{numcases}{}
\nonumber \frac{\dif \bm{X}(t, \alpha)}{\dif t} = \bm{u}(\bm{X}(t, \alpha), t) & $\text{in } \Omega \times [0, T]$\\
\nonumber \bm{X}(0, \alpha) = \alpha & $ \forall \alpha \, \text{ in } \Omega$,
\end{numcases}
 since ${q}(x, t)$ satisfies the transport equation $\eqref{transport-1}$, we have
\begin{equation*}
{q}(\bm{X}(t, \alpha), t) = q_0(\alpha),
\end{equation*}
 then it is easy to see $(\ref{remark-equation-1})$ holds.
\end{remark}
\begin{remark}
Let $s \geq 1$, $q_0 \in H^{s}(\Omega)$ and $\bm{u} \in L^\infty([0, T]; \bm{H}^{s+2}(\Omega))$, then it follows that for all $t\in [0,T]$
\begin{equation} \label{high-transport-estimate}
\|q(t)\|_{H^{s}(\Omega)} \leqslant C \|q_0\|_{H^{s}(\Omega)},
\end{equation}
 where $C$ depends on $\|\bm{u}\|_{L^\infty([0, T]; \bm{H}^{s+2}(\Omega))}$.

Indeed, differentiate equation (\ref{transport-1}) $\alpha$ times, we obtain:
\begin{equation}
\nonumber \partial_t D^\alpha q + \textbf{u} \cdot \nabla{ (D^\alpha q)} + \sum_{\beta < \alpha, \beta + \gamma = \alpha} D^\gamma u \cdot \nabla (D^\beta q) = 0
\end{equation}
then multiply the above equation by $D^\alpha q$, integrate over $\Omega$ and sum up from $\alpha = 0$ to $|\alpha| = s$, we obtain a priori estimates:
\begin{equation*}
\nonumber \frac{\dif}{\dif t} \|q(t)\|_{H^{s}(\Omega)}^2 \leqslant C \|\bm{u}\|_{L^\infty([0, T];\bm{H}^{s+2}(\Omega))} \|q(t)\|_{H^{s}(\Omega)}^2
\end{equation*}
for $t \in [0, T]$, thanks to Gr\"onwall inequality, we have \eqref{high-transport-estimate} holds.
\end{remark}
As we know, a smooth irrotational vector field in simple connected domain is a gradient field of some scalar function. Although two-dimensional exterior domain is not simple connected,  the same conclusion holds true from the following lemma in suitable function space.
\begin{lemma}[\label{lemma-harmonic}see \cite{hieber2021helmholtz}] The following equation
\begin{numcases}{}
\nonumber \nabla \cdot \textbf{u} = 0 & $\text{in } \Omega $ \label{harmonic-1}\\
\nonumber \nabla^\perp \cdot  \textbf{u} = 0 & $\text{in } \Omega $ \label{harmonic-2}\\
\nonumber \textbf{u} \cdot \nu = 0 & $\text{on } \Gamma$ \label{harmonic-3}
\end{numcases}
only has zero solution in $L^2(\Omega) \,space$, in other words,  $X_{har}^2(\Omega) = \lbrace 0 \rbrace$.
\end{lemma}
\begin{remark}{\label{remark-harmonic}}
 Suppose that $\bm{u} \in (L^2(\Omega))^2$ is irrotational, then there exists a scalar function $p \in L^2_{loc}(\Omega)$ such that $\bm{u} = \nabla p$. Indeed, from the classical Helmholtz decomposition, we have
\begin{align}
\bm{u} = \bm{v} + \nabla p
\end{align}
for some $\bm{v} \in L^2_{\sigma}$ and $p \in \bm{\dot{H}}$. Recalling that $\bm{u}$ is irrotational, it follows $\nabla^\perp \cdot \bm{v} = 0$, Lemma $\ref{lemma-harmonic}$ then implies $\bm{v} \equiv 0$.
\end{remark}

Observing that the equation \eqref{vstream3} is the Stokes equation. In the following lemma, we state some well-known results about Stokes equations in exterior domain.
\begin{lemma}[{\label{lemma-stokes}}see \cite{borchers1993boundedness, galdi2011introduction}]
Let $\varphi  \in L^2_{\sigma}(\Omega) \cap \bm{H}^1(\Omega), \lambda > 0$, the following Stokes equation
\begin{numcases}{}
            \bm{u}(x) - \lambda \Delta \bm{u}(x) + \nabla p = \varphi(x) &  $\text{in } \Omega \label{stokes-1} $\\
             \bm{u} = 0 & $\text{on } \Gamma$ \label{stokes-2}
\end{numcases}
has a unique solution $\displaystyle \bm{u} \in \bm{H}^3(\Omega) \cap V$ with the estimate
\begin{equation}\label{stokes-3}
     \begin{split}
      \|\bm{u}\|_{\bm{H}^3(\Omega)} \leqslant C \|\varphi\|_{\bm{H}^1(\Omega)},
     \end{split}
\end{equation}
 where $C$ is a constant depends only on $\Omega$ and $\lambda$.
\end{lemma}

\section{Technical lemmas and their proofs} \label{preliminaries}

In this section, we will state and prove technical lemmas that are helpful to investigate the existence of Euler-$\alpha$ equations in exterior domain.

Firstly, we will discuss on the Poisson equation \eqref{vstream2}. As we know, there are many classical results to Poisson equation in a bounded domain with the Dirichlet boundary condition:
\begin{numcases}{}
\Delta \psi = q & $ \text{in } \,\Omega$ \label{poisson-1} \\
\psi = 0 & $ \text{on } \, \Gamma$. \label{poisson-2}
\end{numcases}
For instance,  in \cite{gilbarg2015elliptic}, Gilbarg and Trudinger established the existence and uniqueness of solutions to the problem above. Furthermore, they showed that
\begin{align}\label{bdd-poisson-estimate}
 \|\psi\|_{\bm{H}^2(\Omega)} \leqslant C \|q\|_{L^2(\Omega)}.
\end{align}
For unbounded domain, there are several results in homogeneous Sobolev spaces(see \cite{samrowski2004poisson}). However, we could not obtain \eqref{bdd-poisson-estimate} in general.

Since $\mathbb{R}^2 \setminus \overline{\Omega}$ is a bounded, open and simple connected domain with $C^\infty$ Jordan boundary, in \cite{iftimie2003two} it has been shown that there exists a smooth biholomorphism $T : \Omega \mapsto \mathbb{R}^2 \setminus \overline{D}$, extending smoothly up to the boundary, mapping $\Gamma$ to $\partial D$. %Furthermore, there also exists a nonzero real number $\beta$ and a bounded holomorphic function $h:\Omega\rightarrow \mathbb{C}$ such that
%\begin{equation}
%\nonumber T(z) = \beta z + h(z)
%\end{equation}
Additionally, one follows that in \cite{iftimie2003two}
\begin{equation}
\begin{aligned}\label{estimate-D_xT}
\|DT\|_{L^\infty} \leqslant C \ &\text{and} \ \|D_xT^{-1}\|_{L^\infty} \leqslant C,\\
|D_x^2T(x)| &=  {O}\left(\frac{1}{|x|^3}\right), \ as \ |x| \rightarrow \infty.
\end{aligned}
\end{equation}
Morever, the authors of \cite{iftimie2003two} have given an explicit formula for the Green's function $G_{\Omega}$ of the Laplacian in $\Omega$, namely
\begin{align}\label{greenfuction}
G_{\Omega}(x, y) = \frac{1}{2\pi} \ln \frac{|T(x) - T(y)|}{|T(x) - T(y)^*|  |T(y)|},
\end{align}
where $\eta^* = \frac{\eta}{|\eta|^2}$ for $\eta \in D^c$. We can now establish the following lemma.
\begin{lemma}\label{lemma-poisson}{}
Let $R > 0$ be fixed. Suppose $q \in L^{2}(\Omega)$ and $\mathrm{supp}\, q \subset B(0, R)$, then the Poisson equations $(\ref{poisson-1})$-$(\ref{poisson-2})$
%\begin{equation}\label{poisson-equation-1}
has a unique solution $\psi \in \bm{\dot{H}}$, which can be written explicitly as
\begin{equation}\label{poisson-4}
\psi(x) = \int_{\Omega}  G_{\Omega}(x, y) q(y) dy
\end{equation}
where $\eta^* = \frac{\eta}{|\eta|^2}   \,in\, D^c$, and $\psi$ obeys
\begin{align}
&\| \nabla \psi \|_{L^2(\Omega)} \leqslant CR(\|q\|_{L^2(\Omega)} + \|q\|_{L^1(\Omega)})\label{poisson-3-1}, \\
&\| D^2 \psi \|_{L^2(\Omega)} \leqslant C(R\|q\|_{L^2(\Omega)} + \|q\|_{L^1(\Omega)}) \label{poisson-3-2},
\end{align}
where $C$ is a constant only depends on $\Omega$. Moreover, if $q \in H^s(\Omega)$ for $s \in \mathbb{N}$, it follows
\begin{equation}\label{esimate-du-m}
\begin{aligned}
\|D^{s+2}\psi\|_{L^2(\Omega)} \leqslant C ( \|q\|_{H^{s}(\Omega)} + \|\nabla \psi\|_{L^2(\Omega)}),
\end{aligned}
\end{equation}
where $C$ also only depends on $\Omega$.
\end{lemma}
\begin{remark}
Since $q$ is compactly supported on $\Omega$, it is clear that the $L^1(\Omega)$-norm of $q$ is bounded by its $L^2(\Omega)$-norm. Here, the $L^1(\Omega)$-norms of $q$ are added, for showing that the coefficients of the inequalities $(\ref{poisson-3-1})$ and $(\ref{poisson-3-2})$   linearly depend on $R$.
\end{remark}
%\begin{remark} Since $T$ is conformal, we know thatic
%\begin{align}
%H_{\Omega}(x) := \ln |T(x)|
%\end{align}
%is harmonic and vanishes on $\Gamma$. Consequently, $\psi + \beta T$ is the solution of Poisson equations $(\ref{poisson-1})$ and $(\ref{poisson-2})$ for arbitrary $\beta \in \mathbb{R}$

%\end{remark}

\begin{proof}[Proof of Lemma \ref{lemma-poisson}]
Let us consider uniqueness first. Suppose $\psi_1, \psi_2 \in \dot{H}$ are both solutions to equations (\ref{poisson-1}) and (\ref{poisson-2}). Set $\bm{w} = \nabla^\perp(\psi_1 - \psi_2)$, it follows
\begin{align}
\nonumber \nabla\cdot \bm{w} = 0 \text{ in } \Omega, \\
\nonumber \nabla^\perp\cdot \bm{w} = 0 \text{ in } \Omega, \\
\nonumber \bm{w} \cdot \nu = 0 \text{ on } \Gamma,
\end{align}{}
where $\nu$ is the unit normal vector to $\Gamma$. From Lemma \ref{lemma-harmonic}, we know that $\bm{w}=0$. Observing that $\psi_1 \equiv \psi_2 \equiv 0$ on $\Gamma$, we conclude that $\psi_1 \equiv \psi_2$ in $\Omega$.

From \cite{iftimie2003two}, one follows that the expression ($\ref{poisson-4}$) is a solution of  equations ($\ref{poisson-1})$ and ($\ref{poisson-2}$). We now check that $\psi\in \bm{\dot{H}}$ with the inequalities ($\ref{poisson-3-1}$), \eqref{poisson-3-2}. For convenience,  zero extension of $q$ on $\mathbb{R}^2\setminus \Omega$ is still denoted by $q$. We will use frequently the following general relation, for arbitrary vectors $a, b \in \mathbb{R}^2$
\begin{align}\label{formula-1}
\left|\frac{a}{|a|^2} - \frac{b}{|b|^2}\right| = \frac{|a - b|} {|a| \ |b|}.
\end{align}
% indeed, squaring the above equality, we have
%\begin{align}
%\nonumber\left|\frac{a}{|a|^2} - \frac{b}{|b|^2}\right|^2 &= \frac{1}{|a|^2} + \frac{1}{|b|^2} - \frac{2a \cdot b}{|a|^2|b|^2}=\frac{|a|^2 + |b|^2 - 2a \cdot b}{|a|^2|b|^2}=\frac{|a - b|^2}{|a|^2|b|^2}
%\end{align}
 %We start to the estimate of $\| \nabla \psi \|_{L^2}$.
 Taking $\bar{R} := \max\{1, \|D_x T\|_{L^\infty}, \|D_x T\|^{-1}_{L^\infty}\} R$, and then noting that $\Omega \equiv \{x\in \mathbb{R}^2 \big| |T(x)| > 1\}$, by ($\ref{formula-1}$), it follows:
\begin{align}
\nonumber \| \nabla \psi\|_{L^2(\Omega)}^2 &= \frac{1}{4\pi^2}\int_{\Omega}\left |\nabla \int_{\Omega} \ln \frac{|T(x) - T(y)|}{|T(x) - (T(y))^*| \ |T(y)|} q(y) \dif y \right|^2 \dif x &\\
\nonumber &\leqslant\frac{1}{4\pi^2}\int_{1 \leqslant |T(x)| \leqslant 2\bar{R}} \left |\int_{\Omega}\frac{D_x(|T(x) - T(y)|)}{|T(x) - T(y)|} q(y) \dif y \right|^2 \dif x &\\
\nonumber & \quad +\frac{1}{4\pi^2}\int_{1 \leqslant |T(x)| \leqslant 2\bar{R}} \left |\int_{\Omega}\frac{D_x(|T(x) - (T(y))^*|)}{|T(x) - (T(y))^*|} q(y) \dif y \right|^2 \dif x & \\
\nonumber & \quad +\frac{1}{4\pi^2}\int_{|T(x)|>2\bar{R}} \left|\int_{\Omega}\frac{|D_xT(x)|\ |T(y) - (T(y))^*|}{|T(x)-T(y)|\ |T(x)-(T(y))^*|} \cdot q(y) \dif y \right|^2 \dif x &\\
 &=: I_1 + I_2 + I_3. & \label{pf-poission-3-1}
\end{align}
 To estimate the term $I_1$, we introduce a truncation function $\chi$, which equals 0 in $D$, while equals 1 in $D^c$. By ($\ref{estimate-D_xT}$), and recalling $q$ is supported on $B(0, R)$, we deduce:
\begin{align}
\nonumber I_1 &=\frac{1}{4\pi^2}\int_{|T(x)| \leqslant 2\bar{R}} \left |\int_{\Omega}\frac{D_x(|T(x) - T(y)|)}{|T(x) - T(y)|} q(y) \dif y \right|^2 \dif x  \\
\nonumber &\leqslant C\int_{|T(x)| \leqslant 2\bar{R}} \left |\int_{\mathbb{R}^2}\frac{|q(y)|}{|T(x) - T(y)|} \dif y \right|^6 \dif x\\
\nonumber &\leqslant C\int_{\mathbb{R}^2} \left |\int_{|T(x) - T(y)| < 3 \bar{R}} |T(x) - T(y)|^{-1}\, |q(y)| \dif y \right|^2 \dif x\\
\nonumber&\leqslant C\int_{\mathbb{R}^2} \left |\int_{\mathbb{R}^2} |T(x)-T(y)|^{-1} \chi\left(\frac{|T(x) - T(y)|}{3\bar{R}}\right) \,|q(y)| \dif y \right|^2 \dif x,
\end{align}
applying the transformation of variables
\begin{align*}
\eta = T(x) \textit{ and } \xi = T(y)
\end{align*}
and  Young's convolution inequality, it follows:
\begin{align}
\nonumber I_1 &\leqslant C\int_{\mathbb{R}^2} \left |\int_{\mathbb{R}^2} |\eta-\xi|^{-1} \chi\left(\frac{|\eta -\xi|}{CR}\right) \,|q(T^{-1}(\xi)| \dif \xi \right|^2 \dif \eta\\
 &\leqslant CR^2 \|q\|_{L^2(\Omega)}^2. \label{I_1}
\end{align}
For the second term $I_2$, observing that $\Omega \equiv \{x\in \mathbb{R}^2 \big| |T(x)| > 1\} \equiv \{x\in \mathbb{R}^2 \big| |(T(x))^*| < 1\}$, we find
\begin{align}
\nonumber I_2 &=\frac{1}{4\pi^2}\int_{1 \leqslant |T(x)| \leqslant 2\bar{R}} \left |\int_{|(T(y))| > 1}\left(\frac{D_x(|T(x) - (T(y))^*|)}{|T(x) - (T(y))^*|} \right) q(y) \dif y \right|^2 \dif x \\
\nonumber &\leqslant \frac{\|DT\|_{L^\infty}}{4\pi^2}\int_{1 \leqslant |T(x)| \leqslant 2\bar{R}} \left |\int_{|(T(y))^*| < 1}\frac{|q(y)|)}{|T(x) - (T(y))^*|}  \dif y \right|^2 \dif x\\
\nonumber &\leqslant \frac{\|DT\|_{L^\infty}}{4\pi^2}\int_{1 \leqslant |T(x)| \leqslant 2\bar{R}} \left |\int_{|(T(y))^*| < 1 \cap |T(x) - (T(y))^*| \leqslant \frac{1}{2}}\frac{|q(y)|)}{|T(x) - (T(y))^*|}  \dif y \right|^2 \dif x\\
\nonumber &\quad + \frac{\|DT\|_{L^\infty}}{4\pi^2}\int_{1 \leqslant |T(x)| \leqslant 2\bar{R}} \left |\int_{|(T(y))^*| < 1 \cap |T(x) - (T(y))^*| \geq \frac{1}{2}}\frac{|q(y)|)}{|T(x) - (T(y))^*|}  \dif y \right|^2 \dif x\\
&=: I_{21}+ I_{22}.
\end{align}
Let us make change of the variables again
\begin{align*}
\eta = T(x) \textit{ and } \xi = (T(y))^*,
\end{align*}
and note that $\dif y = |DT^{-1}((\xi)^*)| \frac{\dif \xi}{|\xi|^4}$, we then get
\begin{align}
\nonumber I_{21} &\leqslant C\int_{1 \leqslant |\eta| \leqslant 2\bar{R}} \left |\int_{|\xi| < 1 \cap |\eta-\xi| \leqslant \frac{1}{2}}\frac{|q(T^{-1}(\xi^*))|)}{|\eta - \xi|}  \frac{\dif \xi}{|\xi|^4} \right|^2 \dif \eta\\
&\leqslant C\int_{1 \leqslant |\eta| \leqslant \frac{3}{2}} \left |\int_{\frac{1}{2} < |\xi| < 1 \cap |\eta-\xi| \leqslant \nonumber \frac{1}{2}}\frac{|q(T^{-1}(\xi^*))|)}{|\eta - \xi|}  \frac{\dif \xi}{|\xi|^4} \right|^2 \dif \eta,
\end{align}
 and then by Young's convolution inequality again, the above inequality implies
\begin{align*}
I_{21} &\leqslant C\|q\|_{L^2(\Omega)}^2.
\end{align*}
The bound of the term $I_{22}$ is easily obtained, indeed,
\begin{align}
\nonumber I_{22} &=\frac{\|DT\|_{L^\infty}}{4\pi^2}\int_{1 \leqslant |T(x)| \leqslant 2\bar{R}} \left |\int_{|(T(y))^*| < 1 \cap |T(x) - (T(y))^*| > \frac{1}{2}}\frac{|q(y)|}{|T(x) - (T(y))^*|}  \dif y \right|^2 \dif x\\
\nonumber &\leqslant C\int_{|T(x)| \leqslant 2\bar{R}} \left |\int_\Omega |q(y)|  \dif y \right|^2 \dif x\\
\nonumber &\leqslant CR^2\|q\|_{L^1(\Omega)}^2.
\end{align}
From the estimates of $I_{21}$ and $I_{22}$, it follows immediately
\begin{align}
\nonumber I_2 &\leqslant CR^2(\|q\|_{L^2(\Omega)}^2 + \|q\|_{L^1(\Omega)}^2).
\end{align}
Now we focus on the term $I_3$. As $q$ is supported on $B(0, R)$, it follows
\begin{align*}
 I_3 &= \frac{1}{4\pi^2}\int_{|T(x)|>2\bar{R}} \left|\int_{\Omega}\frac{|D_xT(x)|\ |T(y) - (T(y))^*|}{|T(x)-T(y)|\ |T(x)-(T(y))^*|} \cdot q(y) \dif y \right|^2 \dif x \\
 &\leqslant \frac{1}{4\pi^2}\int_{|T(x)|> 2\bar{R}} \left|\int_{\Omega} \frac{(\bar{R}+1) \cdot |D_xT(x)|}{\frac{1}{2}|T(x)|^2} \cdot q(y) \dif y \right|^2 \dif x\leqslant C\|q\|_{L^1(\Omega)}^2.
\end{align*}
where we have taken account of $(\ref{estimate-D_xT})$. Collecting the estimates of $I_1, I_2$ and $I_3$, one arrives at the inequality \eqref{poisson-3-1}.

It remains to check the estimate of  $\|D^2\psi\|_{L^2}$. Noting again that $\Omega \equiv \{ x \in \mathbb{R}^2 | |T(x)| > 1\}$, one infers
\begin{align*}
\|D_{ij}\psi\|_{L^2}^2 &=\frac{1}{4\pi^2} \int_{\Omega}\left | \partial_{x_i} \partial_{x_j} \int_{\Omega} \ln \frac{|T(x) - T(y)}{|T(x) - (T(y))^*| \ |T(y)|} q(y) \dif y \right |^2 \dif x \\
&\leqslant \frac{1}{4\pi^2}\int_{\Omega}\left | \lim_{\epsilon \rightarrow 0}\int_{|T(x) - T(y)| = \epsilon} \frac{\partial_{x_i}|T(x) - T(y)|}{|T(x)-T(y)|}  \nu_j(x) \dif S(x) \right |^2 |q(y)|^2\dif y \\
 &+\frac{1}{4\pi^2}\int_{\Omega}\left|\int_{\Omega} \partial_{x_j}\left( \frac{\partial_{x_i}|T(x) - T(y)|}{|T(x)-T(y)|} \right) q(y) \dif y \right|^2 \dif x\\
 &+\frac{1}{4\pi^2}\int_{\Omega}\left|\int_{\Omega} \partial_{x_j}\left( \frac{\partial_{x_i}|T(x) - (T(y))^*|}{|T(x)-(T(y))^*|} \right) q(y) \dif y \right|^2 \dif x\\
 &=: J_1 + J_2 + J_3. \label{D_ij}
\end{align*}
We will examine  the above terms one by one. Since $T$ is conformal, it is clear to know $\dif S(T^{-1}(\xi)) \equiv |DT^{-1}(\xi)|\dif S(\xi)$. Due to $(\ref{estimate-D_xT})$, the term $J_1$ is bounded by
\begin{align*}
\nonumber J_1 &=  \frac{1}{4\pi^2}\int_{\Omega}\left | \lim_{\epsilon \rightarrow 0}\int_{|T(x) - T(y)| = \epsilon} \frac{\partial_{x_i}|T(x) - T(y)|}{|T(x)-T(y)|}  \nu_j(x) \dif S(x) \right |^2 |q(y)|^2\dif y \\
\nonumber &\leqslant   \frac{1}{4\pi^2}\int_{\Omega}\left | \lim_{\epsilon \rightarrow 0}\int_{|T(x) - T(y)| = \epsilon} \frac{|DT(x)|}{|T(x)-T(y)|} \dif S(x) \right |^2  |q(y)|^2\dif y\\
\nonumber &\leqslant C \|q\|_{L^2(\Omega)}^2.
\end{align*}
From the transformation of variables
\begin{align*}
\eta = T(x) \textit{ and } \xi = T(y),
\end{align*}
we then have checked the term $J_2$ as follows
\begin{align*}
  J_2 &=\frac{1}{4\pi^2}\int_{D^c} \left | \int_{D^c} K(\eta, \eta-\xi) Q_2(\xi)\dif \xi \right|^2 \dif \eta +\frac{1}{4\pi^2}\int_{D^c} \left | \int_{D^c} N(\eta, \eta-\xi) Q_2(\xi)\dif \xi \right|^2 \dif \eta\\
\nonumber &=: J_{21} + J_{22}
\end{align*}
with
\begin{equation*}
\begin{aligned}
 K(\eta, \xi) &:= \partial_{\xi_k} \left(\frac{\partial_{\xi_l}|\xi|} {|\xi|}  \right) D_iT_l(T^{-1}(\eta)) D_jT_k(T^{-1}(\eta)) |DT^{-1}(\eta)|, \\
N(\eta, \xi) &:= \frac{\partial_{\xi_l}|\xi|}{|\xi|} \partial_{ij}T_l(T^{-1}(\eta)) |DT^{-1}(\eta)|,\\
 Q_2(\xi) &:= q(T^{-1}(\xi)) |DT^{-1}(\xi)|.
\end{aligned}
\end{equation*}
One is able to check that $K$ is a singular kernel satisfying the assumptions of Calder\'on-Zygmund Theorem, then one follows
 \begin{equation}
 \nonumber J_{21} \leqslant C \|Q_2\|_{L^2(\Omega)}^2 \leqslant C \|q\|_{L^2(\Omega)}^2.
 \end{equation}
  Recalling that $D_{ij}T(x)$ is bounded in $\Omega$ and satisfies $D_{ij}T(x) = {O}(|x|^{-3})$ for $|x| \rightarrow \infty$, we arrive at
\begin{equation}
\begin{aligned}
\nonumber J_{22} &= \frac{1}{4\pi^2}\int_{D^c} \left | \int_{D^c} N(\eta, \eta-\xi) Q_2(\xi)\dif \xi \right|^2 \dif \eta\\
\nonumber&\leqslant C\left(\int_{|\eta| > 2\bar{R}} \left | \int_{D^c} \frac{1}{|\eta -\xi||\eta|^3}Q_2(\xi)\dif \xi \right|^2 \dif \eta + \int_{|\eta| < 2\bar{R}} \left | \int_{D^c} \frac{1}{|\eta -\xi|}Q_2(\xi)\dif \xi \right|^2 \dif \eta\right).
\end{aligned}
\end{equation}
Since $Q_2$ is supported on $B(0, \bar{R})$, by Young's convolution inequality, we deduce
\begin{equation*}
\begin{aligned}
\nonumber J_{22} &\leqslant C\left(\int_{|\eta| > 2\bar{R}} |\eta|^{-8} \left[ \int_{D^c} |Q_2(\xi)|\dif \xi \right]^2 \dif \eta + \int_{|\eta| < 2\bar{R}} \left | \int_{D^c} \frac{\chi(\frac{\eta-\xi}{3\bar{R}})}{|\eta -\xi|}Q_2(\xi)\dif \xi \right|^2 \dif \eta\right)\\
\nonumber&\leqslant C(R^2\|q\|_{L^2(\Omega)}^2 + \|q\|_{L^1(\Omega)}^2).
\end{aligned}
\end{equation*}
 From the estimates of $J_{21}$ and $J_{22}$,  one yields
\begin{align}
\nonumber J_2 \leqslant C(R^2\|q\|_{L^2(\Omega)}^2 + \|q\|_{L^1(\Omega)}^2).
\end{align}
As
\begin{align}
\nonumber \eta = T(x) \textit{ and } \xi = (T(y))^*,
\end{align}
 we have
\begin{align}
\nonumber J_3 &=\frac{1}{4\pi^2}\int_{\Omega}\left|\int_{\Omega} \partial_{x_j}\left( \frac{\partial_{x_i}|T(x) - (T(y))^*|}{|T(x)-(T(y))^*|} \right) q(y) \dif y \right|^2 \dif x\\
\nonumber &=\int_{D^c} \left | \int_{D} K(\eta, \eta-\xi) Q_3(\xi)\dif \xi \right|^2 \dif \eta +\int_{D^c} \left | \int_{D} N(\eta, \eta-\xi) Q_3(\xi)\dif \xi \right|^2 \dif \eta\\
\nonumber &\leqslant\frac{1}{4\pi^2}\int_{D^c} \left | \int_{\frac{1}{2} < |\xi| < 1} K(\eta, \eta-\xi) Q_3(\xi)\dif \xi \right|^2 \dif \eta + \frac{1}{4\pi^2}\int_{D^c} \left | \int_{|\xi|<\frac{1}{2}} K(\eta, \eta-\xi) Q_3(\xi)\dif \xi \right|^2 \dif \eta\\
\nonumber &\quad +\frac{1}{4\pi^2}\int_{D^c} \left | \int_{\frac{1}{2} <|\xi| < 1} N(\eta, \eta-\xi) Q_3(\xi)\dif \xi \right|^2 \dif \eta +\frac{1}{4\pi^2}\int_{D^c} \left | \int_{|\xi| < \frac{1}{2}} N(\eta, \eta-\xi) Q_3(\xi)\dif \xi \right|^2 \dif \eta\\
\nonumber &=: J_{31} + J_{32} + J_{33} + J_{34}
\end{align}
with
\begin{align}
\nonumber Q_3(\xi) = \frac{|DT^{-1}(\xi^*)|q(T^{-1}(\xi^*))}{|\xi|^4}.
\end{align}
By Calder\'on-Zygmund Theorem, we obtain
\begin{align}
\nonumber J_{31} &= \frac{1}{4\pi^2}\int_{D^c} \left | \int_{\frac{1}{2} < |\xi| < 1} K(\eta, \eta-\xi) Q_3(\xi)\dif \xi \right|^2 \dif \eta\\
\nonumber &\leqslant C \int_{\frac{1}{2} < |\xi| < 1} |Q_3(\xi)|^2\dif \xi\leqslant C \|q\|_{L^2(\Omega)}^2.
\end{align}
Let us now proceed the term $J_{32}$ as
\begin{align}
\nonumber J_{32} &= \frac{1}{4\pi^2}\int_{D^c} \left | \int_{|\xi| < \frac{1}{2}} K(\eta, \eta-\xi) Q_3(\xi)\dif \xi \right|^2 \dif \eta\\
\nonumber &\leqslant C \int_{D^c} \left | \int_{|\xi| < \frac{1}{2}} \frac{1}{|\eta-\xi|^2}|Q_3(\xi)|\dif \xi\right|^2 \dif \eta\\
\nonumber &\leqslant C \int_{D^c} \left | \int_{D}|Q_3(\xi)|\dif \xi\right|^2 \frac{1}{|\eta|^4}\dif \eta \\
\nonumber &\leqslant C \|q\|_{L^1(\Omega)}^2.
\end{align}
We observe that the following fact for $J_{33}$
\begin{align*}
 J_{33} &= \frac{1}{4\pi^2}\int_{D^c} \left | \int_{\frac{1}{2} < \xi < 1} N(\eta, \eta-\xi) Q_3(\xi)\dif \xi \right|^2 \dif \eta\\
&\leqslant C\int_{1 <|\eta| < 2\tilde{R} } \left | \int_{\frac{1}{2} < \xi < 1} N(\eta, \eta-\xi) Q_3(\xi)\dif \xi \right|^2 \dif \eta \\
&\quad +C\int_{|\eta| > 2\tilde{R}} \left | \int_{\frac{1}{2} < \xi < 1} N(\eta, \eta-\xi) Q_3(\xi)\dif \xi \right|^2 \dif \eta\\
&\leqslant C\int_{1 <|\eta| < 2\tilde{R} } \left | \int_{\frac{1}{2} < \xi < 1} \frac{1}{|\eta - \xi|} Q_3(\xi)\dif \xi \right|^2 \dif \eta \\
&\quad+C\int_{|\eta| > 2\tilde{R}} \left | \int_{\frac{1}{2} < \xi < 1} \frac{1}{|\eta-\xi|\,|\eta|^2} Q_3(\xi)\dif \xi \right|^2 \dif \eta.
\end{align*}
By Young's convolution inequality to the first term of the above inequality, we conclude that
\begin{align}
\nonumber J_{33} \leqslant C(\|q\|_{L^1(\Omega)}^2 + R^2\|q\|_{L^2(\Omega)}^2).
\end{align}
For $J_{34}$, it turns out to show
\begin{align*}
 J_{34} &= \frac{1}{4\pi^2}\int_{D^c} \left | \int_{|\xi|<\frac{1}{2}} N(\eta, \eta-\xi) Q_3(\xi)\dif \xi \right|^2 \dif \eta \\
 &\leqslant C\int_{D^c} \left [ \int_{|\xi|<\frac{1}{2}} \frac{1}{|\eta-\xi|\,|\eta|^2} Q_3(\xi) \dif \xi \right]^2\dif \eta\\
  &\leqslant C\|q\|_{L^1(\Omega)}^2.
\end{align*}
Collecting the estimates of $J_{31}$, $J_{32}$, $J_{33}$ and $J_{34}$, we find
\begin{align*}
 J_3 \leqslant C(\|q\|_{L^1(\Omega)}^2 + R^2\|q\|_{L^2(\Omega)}^2).
\end{align*}
All these estimates for $J_1$, $J_2$ and $J_3$ imply
\begin{align}
\|D^2\psi\|_{L^2(\Omega)} \leqslant C(\|q\|_{L^1(\Omega)} + R\|q\|_{L^2(\Omega)}), \label{estimate-J}
\end{align}
which is exactly the inequality \eqref{poisson-3-2}.

At last, we check that the inequality \eqref{esimate-du-m} holds. Taking $L > 0$ such that $\mathbb{R}^2 \setminus \overline{\Omega} \subset\subset B(0, L/2)$, and then choosing $\phi \in C^\infty(\mathbb{R}^2)$ which is a nonnegative function such that $\phi(x) = 0$ for $|x| < L/2$ and $\phi(x)= 1$ for $|x| > L$, one finds that $\phi \psi$ obeys
\begin{equation*}
\nonumber \Delta (\phi \psi) = \phi q + 2 \nabla \phi \cdot \nabla \psi + \Delta \phi \psi \textit{ in } \mathbb{R}^2,
\end{equation*}
therefore, it follows
\begin{equation*}
\begin{aligned}
\nonumber \|D^{s+2} \psi\|_{L^2(\Omega^{L})} &\leqslant \|D^{s+2}(\phi\psi)\|_{L^2(\mathbb{R}^2)} \\
&\leqslant C (\|q\|_{H^s(\Omega)} + \|\psi\|_{H^{s+1}(\Omega_{L})}).
\end{aligned}
\end{equation*}
Noting that $\psi = 0$ on $\Gamma$, by Sobolev interpolation inequality and Poincar\'e inequality, we follow
\begin{equation}\label{poisson-high-1}
\begin{aligned}
\|D^{s+2} \psi\|_{L^2(\Omega^{L})} \leqslant C (\|q\|_{H^s(\Omega)} + \|\nabla\psi\|_{L^2(\Omega)} + \|D^{s+1}\psi\|_{L^2(\Omega)}).
\end{aligned}
\end{equation}
Now it is sufficient to consider the Poisson equation in bounded domain $\Omega_L$. Let us denote $\Pi_{L, 2L} = \{ x \in \mathbb{R}^2 \big| L < |x| < 2L\}$, we arrive at
\begin{equation}\label{poisson-high-tmp}
\begin{aligned}
 \|D^{s+2}\psi\|_{L^2(\Omega_L)} &\leqslant C(\|q\|_{H^s(\Omega)} + \|\psi\|_{H^{s + 2-\frac{1}{2}}(\partial B(0, L))})\\
&\leqslant C(\|q\|_{H^s(\Omega)} + \|\psi\|_{\bm{H}^{s+2}(\Pi_{L, 2L})}).
\end{aligned}
\end{equation}
In view of Gagliardo-Nirenberg  inequality, it follows that
\begin{align}\label{poisson-high-tmp2}
 \|\psi\|_{\bm{H}^{s+2}(\Omega)(\Pi_{L, 2L})} \leqslant C (\|\psi\|_{L^2(\Omega_{2L})} + \|D^{s+2}\psi\|_{L^2(\Omega^L)}),
\end{align}
then by Poincar\'e Inequality, $(\ref{poisson-high-tmp})$ and $(\ref{poisson-high-tmp2})$ immediately bring out
\begin{equation}
\begin{aligned}
\|\psi\|_{\bm{H}^{s+2}(\Omega)(\Pi_{L, 2L})} &\leqslant C(\|q\|_{H^s(\Omega)} + \|\nabla\psi\|_{L^2(\Omega)} + \|D^{s+2}\psi\|_{L^2(\Omega^L)}).\label{poisson-high-2}
\end{aligned}
\end{equation}
From \eqref{poisson-high-1} and \eqref{poisson-high-2}, we obtain that
\begin{align}\label{tmp-poisson-high-2}
\|D^{s+2}\psi\|_{L^2(\Omega)} &\leqslant C(\|q\|_{H^s(\Omega)} + \|D^{s+1}\psi\|_{L^2(\Omega)} + \|\nabla \psi\|_{L^2(\Omega)}),
\end{align}
consequently, we can deduce (\ref{esimate-du-m}) by induction.
\end{proof}

As we know, in simple connected domain of $\mathbb{R}^2$, a divergence free vector field $\bm{u}$ is related to a stream function $\psi$, i.e. $\bm{u} = \nabla^\perp \psi$. For exterior domain of $\mathbb{R}^2$, although it is not simple connected, we still have similar result.
\begin{lemma}\label{stream-function-lemma} Let the vector field $\bm{u} \in L_\sigma^2 \cap \bm{H}^1(\Omega)$, then there exists a scalar function $\psi$ such that $\bm{u} \equiv \nabla^\perp \psi$.
\end{lemma}

\begin{proof}
As above,  the Green function of the Laplacian in $\Omega$ is given by \eqref{greenfuction}, and we now show that the following function
\begin{equation*}
\psi(x) := \int_{\Omega} G_{\Omega}(x, y)  \nabla^\perp_y\cdot \bm{u}(y) \dif y
\end{equation*}
is the stream function function of $\bm{u}$ with  $\bm{u} \equiv \nabla^\perp \psi$. Let $\bm{w} = \bm{u} - \nabla^\perp \psi$. In view of Lemma $\ref{lemma-harmonic}$, it suffices to verify that $\bm{w}$ satisfies
\begin{itemize}
    \item[(i)] $\nabla \cdot \bm{w} = 0 \text{ in } \Omega$;
  \item[(ii)] $\nabla^\perp \cdot \bm{w} = 0 \text{ in } \Omega$;
  \item[(iii)] $\bm{w} \cdot \nu \equiv 0 \text{ on } \Gamma$;
    \item[(iv)] $\bm{w} \in L^2(\Omega)$,
\end{itemize}
where $\nu$ is the normal vector to $\Gamma$. It is easy to check that condition (i)-(iii) hold. We here only verify the condition (iv), equivalently, we need to check $\nabla \psi \in L^2(\Omega)$. Indeed, $\forall \phi \in C_0^\infty(\Omega)$, by integrating by parts, it follows
\begin{equation*}
\begin{aligned}
(\partial_i \psi, \phi)_{L^2} &= - (\psi, \partial_i \phi)_{L^2}\\
&= -\int_{\Omega} \left[\int_{\Omega} G_{\Omega}(y, x) \nabla^\perp_y \cdot \bm{u}(y) \dif y \right]\partial_i\phi(x) \dif x \\
&= \int_{\Omega} \left[\int_{\Omega} \left[\nabla^\perp_y G_{\Omega}(y, x)\right] \cdot \bm{u}(y) \dif y \right]\partial_i\phi(x) \dif x \\
&=  \int_{\Omega} \left[\int_{\Omega} \partial_i \phi(x) \left[ \nabla^\perp_y G_{\Omega}(y, x)\right] \dif x \right]\cdot \bm{u}(y) \dif y \\
&=  \frac{1}{4\pi^2}\int_{\Omega} \left[\lim_{\epsilon \rightarrow 0} \int_{|T(x) - T(y)| = \epsilon} \nu_i \phi(x) \left[ \nabla^\perp_y \ln|T(y) - T(x)|\right] \dif S(x) \right]\cdot \bm{u}(y) \dif y \\
&\quad- \frac{1}{4\pi^2}\int_{\Omega} \left[\int_{\Omega} \phi(x) \left[ \partial_{x_i} \nabla^\perp_y \ln|T(y) - T(x)|\right] \dif x \right]\cdot \bm{u}(y) \dif y \\
&\quad- \frac{1}{4\pi^2}\int_{\Omega} \left[\int_{\Omega} \phi(x) \left[ \partial_{x_i} \nabla^\perp_y \ln|T(y) - (T(x))^*|\right] \dif x \right]\cdot \bm{u}(y) \dif y \\
&=: L_1 + L_2 + L_3,
\end{aligned}
\end{equation*}
where  the property that $G(x,y) = G(y,x)$ in $\Omega$ was used.   From $\|DT\|_{L^\infty} < \infty$ and H{\"o}lder inequality, The term $L_1$ can be bounded
\begin{equation}\label{estimate_l_1}
\begin{aligned}
|L_1| &= \left|\frac{1}{4\pi^2}\int_{\Omega} \left[\lim_{\epsilon \rightarrow 0} \int_{|T(x) - T(y)| = \epsilon} \nu_i \phi(x) \left[ \nabla^\perp_y \ln|T(y) - T(x)|\right] \dif S(x) \right]\cdot \bm{u}(y) \dif y \right|\\
&\leqslant \frac{1}{4\pi^2}\int_{\Omega} \left[\lim_{\epsilon \rightarrow 0} \int_{|T(x) - T(y)| = \epsilon} \phi(x) \frac{|DT(y)|}{|T(y) - T(x)|} \dif S(x)\right] |\bm{u}(y)| \dif y\\
&\leqslant C\int_{\Omega} |\phi(y)| |\bm{u}(y)| \dif y\\
&\leqslant C\|\phi\|_{L^2(\Omega)} \|\bm{u}\|_{L^2(\Omega)}.
\end{aligned}
\end{equation}
by making change of variable, the second term $L_2$ can be written by
\begin{equation*}
\begin{aligned}
L_2 &= -\frac{1}{4\pi^2}\int_{\Omega} \left[\int_{\Omega} \phi(x) \left[ \partial_{x_i} \nabla^\perp_y \ln|T(y) - T(x)|\right] \dif x \right]\cdot \bm{u} \dif y \\
&= -\frac{1}{4\pi^2}\int_{\Omega}\left[\int_{\Omega} \phi(x)\partial_{x_i} \frac{\nabla^\perp_y |T(y) - T(x)|}{|T(y) - T(x)|}\dif x\right] \cdot \bm{u}(y) \dif y\\
&\xlongequal[\xi = T(y)]{\eta = T(x)}-\frac{1}{4\pi^2} \int_{D^c}\int_{D^c} \bm{K}_{kl}(\xi-\eta) \bm{\Phi}_{ik}(\eta)\bm{U}_{l}(\xi) \dif \eta \dif \xi
\end{aligned}
\end{equation*}
with
\begin{equation}\label{kernel_lemma_approximation}
\begin{aligned}
\bm{K}_{kl}(\eta) &:=  -\partial_{\eta_k}\left[\frac{\partial_{\eta_l}|\eta|}{|\eta|}\right]\\
\bm{\Phi}_{ik}(\eta) &:= \phi(T^{-1}(\eta)) (\partial_i T_k)(T^{-1}(\eta))|DT^{-1}(\eta)|\\
 \bm{U}_{l}(\xi) &:= (\nabla^\perp T_l)(T^{-1}(\xi))\cdot \bm{u}(T^{-1}(\xi)) |DT^{-1}(\xi)|.
\end{aligned}
\end{equation}
As $\bm{K}_{kl}$ is singular kernel satisfying Calder\'on-Zygmund Theorem, it is easy to obtain
\begin{equation*}
\begin{aligned}
|L_2| \leqslant C\|\bm{\Phi}\|_{L^2(\Omega)}\|\|\bm{U}\|_{L^2(\Omega)},
\end{aligned}
\end{equation*}
Taking account of property ($\ref{estimate-D_xT}$), the above inequality then implies
\begin{equation}\label{estimate_l_2}
\begin{aligned}
|L_2| \leqslant C\|\phi\|_{L^2(\Omega)}\|\|\bm{u}\|_{L^2(\Omega)}.
\end{aligned}
\end{equation}
Similarly, we change the term $L_3$ into the following formula
\begin{equation*}
\begin{aligned}
L_3 &= -\frac{1}{4\pi^2}\int_{\Omega} \left[\int_{\Omega} \phi(x) \left[ \partial_{x_i} \nabla^\perp_y \ln|T(y) - (T(x))^*|\right] \dif x \right]\cdot \bm{u}(y) \dif y \\
&= -\frac{1}{4\pi^2}\int_{\Omega}\left[\int_{\Omega} \phi(x)\partial_{x_i} \frac{\nabla^\perp_y |T(y) - (T(x))^*|}{|T(y) - T(x)|}\dif x\right] \cdot \bm{u}(y) \dif y\\
&\xlongequal[\xi = T(y)]{\eta = (T(x))^*} -\frac{1}{4\pi^2}\int_{D^c}\int_{D} \bm{K}_{kl}(\xi-\eta) \Theta_{ik}(\eta)\bm{U}_{l}(\xi) \dif \eta \dif \xi
\end{aligned}
\end{equation*}
where $\bm{K}$ and $\bm{U}$ are defined in (\ref{kernel_lemma_approximation}), while $\Theta$ is defined as
\begin{equation*}
\begin{aligned}
\Theta_{ik}(\eta) = \phi(T^{-1}(\eta^*))\frac{|DT^{-1}(\eta^*)|}{|\eta|^2}(\partial_iT_k(T^{-1}(\eta^*)) - 2 |\eta|^2\eta^*_k\eta^*_j\partial_iT_j(T^{-1}(\eta^*))).
\end{aligned}
\end{equation*}
Again by Calder\'on-Zygmund Theorem and \eqref{estimate-D_xT}, we arrive at
\begin{equation}\label{stream-funtion-L-3-1}
\begin{aligned}
|L_3| &\leqslant C\|\Theta\|_{L^2(D)}\|\|\bm{U}\|_{L^2(D^c)}\\
&\leqslant C\|\Theta\|_{L^2(D)}\|\|\bm{u}\|_{L^2(\Omega)}.
\end{aligned}
\end{equation}
 We now show that $\Theta \in L^2(D)$, indeed,
\begin{equation}\label{stream-funtion-L-3-2}
\begin{aligned}
\|\Theta\|_{L^2(D)}^2 &= \int_D |\Theta(\eta)|^2 \dif \eta \\
&= \int_D \left|\phi(T^{-1}(\eta^*))\frac{|DT^{-1}(\eta^*)|}{|\eta|^2}(\partial_iT_k(T^{-1}(\eta^*)) - 2 |\eta|^2\eta^*_k\eta^*_j\partial_iT_j(T^{-1}(\eta^*)))\right|^2 \dif \eta\\
&\leqslant C\int_D \frac{|\phi(T^{-1}(\eta^*))|^2}{|\eta|^4}  \dif \eta.\\
\end{aligned}
\end{equation}
 By variable substitution, it follows
\begin{equation}\label{stream-funtion-L-3-3}
\begin{aligned}
\int_D \frac{|\phi(T^{-1}(\eta^*))|^2}{|\eta|^4}  \dif \eta &\xlongequal{\eta = (T(x))^*}\int_{\Omega} \frac{|\phi(x)|^2}{|T^*(x)|^4} \frac{1}{|T(x)|^4} |DT(x)| \dif x\\
&\leqslant C \int_\Omega |\phi(x)|^2 \dif x.
\end{aligned}
\end{equation}
By ($\ref{stream-funtion-L-3-1}$), ($\ref{stream-funtion-L-3-2}$) and ($\ref{stream-funtion-L-3-3}$), we thus obtain
\begin{equation}\label{estimate_l_3}
|L_3| \leqslant C \|\phi\|_{L^2(\Omega)}\|\bm{u}\|_{L^2(\Omega)}.
\end{equation}
Collecting (\ref{estimate_l_1}), (\ref{estimate_l_2}) and (\ref{estimate_l_3}), we conclude that
\begin{equation*}
|(\partial_i \psi, \phi)_{L^2(\Omega)}| \leqslant C\|\phi\|_{L^2(\Omega)} \|\bm{u}\|_{L^2(\Omega)}.
\end{equation*}
Therefore the condition (iv) holds, in virtue of Lemma \ref{lemma-harmonic}, we have $\bm{u} \equiv \nabla^\perp \psi$.
\end{proof}

In this article, we will first consider the special case that the initial \textit{unfiltered} is compactly supported. The following lemma helps to construct such an approximate sequence for more general initial data $\bm{u}_0$.
\begin{lemma}\label{approximate-lemma} Assume that $\bm{u} \in V \cap \bm{H}^s(\Omega)$, $s \geq 1$. Then there exists a approximate sequence $\{\bm{u}^n\} \subset V \cap \bm{H}^s(\Omega)$ such that $\bm{u}^n $ is compactly supported and converges to $\bm{u}$ in $\bm{H}^s(\Omega)$ strongly.
\end{lemma}

\begin{proof}
Without loss of generality, we assume $\mathbb{R}^2 \setminus \overline{\Omega} \subset\subset B(0, 1)$. Now let $\phi \in C_0^\infty(\mathbb{R}^2)$ be a non-negative function such that
\begin{itemize}
  \item[(i)] $0 \leqslant \phi \leqslant 1$ in $\mathbb{R}^2$;
  \item[(ii)] $\phi \equiv 1$ in $B(0,1)$;
  \item[(iii)] $\phi \equiv 0$ in $\mathbb{R}^2 \setminus \overline{B(0, 2)}$.
\end{itemize}
We will use the following notations in the following context
\begin{equation*}
\begin{aligned}{}
\phi^n(x) := \phi(\frac{x}{n}),\\
\Pi_{n, 2n} := \{ x \in \mathbb{R}^2 \big| n < |x| < 2n \}.
\end{aligned}
\end{equation*}
Let $\psi$ be the stream function of $\bm{u}$ constructed in lemma \ref{stream-function-lemma}, we define the approximate sequence $\{\bm{u}^n\}$ as
\begin{equation*}
\bm{u}^n := \nabla ^ \perp (\phi^n (\psi - a_n)),
\end{equation*}
with
\begin{equation*}
a_n = \frac{1}{|\Pi_{n, 2n}|}\int_{\Pi_{n, 2n}} \psi(x) \dif x.
\end{equation*}
Noting that $\phi$ is compactly supported, it follows $\bm{u}^n$ is compactly supported. It suffices to prove that $\bm{u}^n$ converges to $\bm{u}$ in $H^s$. Since
\begin{equation*}
\begin{aligned}
\nonumber \|\bm{u}^n - \bm{u}\|_{H^s(\Omega)}^2 &= \sum_{|\alpha| \leqslant s} \|D^\alpha \bm{u}^n - D^\alpha\bm{u}\|_{L^2}^2 \\
\nonumber &= \sum_{|\alpha| \leqslant s} \| D^\alpha (\nabla^\perp (\phi^n (\psi-a_n))) - D^\alpha\bm{u}\|_{L^2}^2 \\
&\leqslant \sum_{|\alpha|\leqslant s} \|(\phi^n - 1)D^\alpha\bm{u}\|_{L^2}^2 + \sum_{1 \leqslant|\alpha| \leqslant s, |\beta| \leqslant s} C_{\alpha, \beta}\|D^\alpha \phi^n D^\beta (\psi-a_n)\|_{L^2(\Omega)}^2,
\end{aligned}
\end{equation*}
it is obvious  to see that the first term of the above inequality converges to zero as $n \rightarrow \infty$. For the second one, we treat $\beta = 0$ and $\beta \neq 0$ respectively. In the case $\beta = (\beta_1, \beta_2)\neq 0$, it follows
\begin{equation*}
\begin{aligned}
\|D^\alpha\phi^n D^\beta (\psi- a_n)\|_{L^2}^2 &\leqslant \|D^\alpha\phi^n D^{|\beta| - 1}\bm{u}\|_{L^2}^2\\
&\leqslant \frac{C_\alpha}{n^{2|\alpha|}} \|\bm{u}\|_{\bm{H}^s(\Omega)}^2.
\end{aligned}
\end{equation*}
As $\alpha > 0$, we deduce that
\begin{align}
 \|D^\alpha\phi^n D^\beta (\psi- a_n)\|_{L^2(\Omega)}^2 \rightarrow 0  \, as \, n \rightarrow \infty. \label{app-lemma-1}
\end{align}
In the case $\beta = 0$, by Poincar\'e inequality, it follows
\begin{equation}
\begin{aligned}
\nonumber \|D^\alpha\phi^n (\psi-a_n)\|_{L^2}^2 &= \int_{\Omega} |D^\alpha \phi^n(x) (\psi(x) -a_n)|^2 \dif x\\
\nonumber &\leqslant \frac{C_\alpha}{n^{2|\alpha|}} \int_{ n \leqslant |x| \leqslant 2n} |\psi(x) - a_n|^2 \dif x\\
\nonumber &\leqslant \frac{C_\alpha}{n^{2|\alpha| - 2}} \int_{ 1 \leqslant |y| \leqslant 2} |\psi(ny) -a_n|^2 \dif y\\
\nonumber &\leqslant \frac{C_\alpha}{n^{2|\alpha|}} \int_{ 1 \leqslant |y| \leqslant 2} |\bm{u}(ny)|^2  \dif y\\
&\leqslant \frac{C_\alpha}{n^{2|\alpha| - 2}} \int_{ n \leqslant |x| \leqslant 2n} |\bm{u}(x)|^2  \dif x.
\end{aligned}
\end{equation}
Since $\bm{u} \in \bm{H}^s(\Omega)$ and $|\alpha| \geq 1 $, the above inequality implies that
\begin{align}
\|D^\alpha\phi^n (\psi -a_n)\|_{L^2(\Omega)}^2 \rightarrow 0 \,~\mbox{as}~ \,n \rightarrow \infty. \label{app-lemma-2}
\end{align}
By $(\ref{app-lemma-1})$ and $(\ref{app-lemma-2})$, it follows that $\bm{u}^n$ converges to $\bm{u}$ in $\bm{H}^s(\Omega)$ strongly.
\end{proof}
\section{Global existence theorems to Euler-$\alpha$ equations in 2D exterior domain}
In this whole section, without loss of generality, we set $\alpha = 1$, and the equations $(\ref{euler-alpha-1})-(\ref{euler-alpha-2})$ with initial data and boundary conditions \eqref{Dirichlet} and \eqref{inftydecay} are been rewritten as
\begin{numcases}{}
  \partial_t \bm{v} + \bm{u} \cdot {\nabla} \bm{v} + (\nabla \bm{u})^t \cdot \bm{v} + \nabla p = 0
    \qquad \qquad  \qquad \qquad  &  $\text{in}  \ \Omega \times (0, T), $ \label{euler-alpha-1r}  \qquad \   \qquad  \qquad \qquad \\
    \text{div} \ \bm{u} = 0 &  $\text{in} \ \Omega \times [0, T) $ \label{euler-alpha-2r},\\
    \bm{u} = 0  & $\text{in} \ \Gamma \times [0, T] \label{euler-alpha-3r}$,\\
    \bm{u}|_{t=0} = \bm{u}_0 & $\text{in} \ \Omega $\label{euler-alpha-4r},\\
    \bm{u} (x, t) \to 0 & $\forall t \in [0, T), |x| \to \infty $ \label{euler-alpha-5r},
\end{numcases}
where $\bm{v} = \bm{u} -\Delta \bm{u}$, $\bm{u}_0$ is the initial $\textit{filtered}$ velocity, and $T > 0$ is arbitrary.
Let us firstly consider the special case that the initial  \textit{unfiltered} vorticity $q_0$ is compactly supported, we have the following proposition.
\begin{prop} \label{proposition-1}
 Assume that any $T>0$ and the initial \textit{filtered} velocity $\textbf{u}_0 \in \bm{H}^3(\Omega) \cap V$ with the initial \textit{unfiltered} vorticity $q_0$ is compactly supported, then the equations $(\ref{euler-alpha-1r})-(\ref{euler-alpha-5r})$ has a unique weak solution $\textbf{u} \in L^\infty(0, T; \bm{H}^3(\Omega))\cap C([0,T]; V)$ in the following sense, for any $\phi \in C^\infty([0, T]; \mathcal{D})$, the indentity
\begin{equation}\label{energy-formula}
\begin{aligned}
 (\bm{u}(t), \phi)_{L^2(\Omega)} + (\nabla \bm{u}(t), \nabla \phi) - (\bm{u}_0, \phi) - (\nabla \bm{u}_0, \nabla \phi)_{L^2(\Omega)} \\
= \int_0^t(\bm{v}, \phi_t)_{L^2(\Omega)} - \int_0^t (\bm{u}\cdot \nabla \bm{v} + (\nabla \bm{u})^t \cdot \bm{v}, \phi)_{L^2(\Omega)}
\end{aligned}
\end{equation}
holds for $t\in[0, T]$, where $\bm{v} = \bm{u} - \Delta\bm{u}$.
\end{prop}
\begin{proof} Let $R_0 > 0$ such that $\supp \, q_0 \subset B(0, R_0)$. Let $T_0 > 0$ be determined later, we  will construct a map $\mathcal{F}$ from $C([0, T_0]; V)$ to itself, and then we will show $\mathcal{F}$ is a contraction map for a short time interval such that we  can use Banach fixed point theorem to prove the local well-posedness of equations $(\ref{euler-alpha-1r})$--$(\ref{euler-alpha-5r})$, finally we will extend the solution to any time interval.

$\bm{Step \ 1}$  Construct the map $\mathcal{F}:C([0, T_0]; V) \rightarrow C(0, T_0]; V)$. The domain of $\mathcal{F}$ is defined as $D(\mathcal{F}) := \{ \bm{u} \in C([0, T_0]; V \cap \bm{H}^3(\Omega) \ \big  | \ \sup_{t\in[0, T_0]}\|\bm{u}(t)\|_{\bm{H}^3} \leqslant M, \bm{u}|_{t=0} = \bm{u}_0\}$, where $M > 0$ is determined later. Let $\bm{u}\in D(\mathcal{F})$, we can find that there exists a unique triple $(q, \psi, \tilde{\bm{u}})$ satisfies the following equations in suitable Sobolev space
\begin{numcases}{}
     \partial_t q + \bm{u} \cdot \nabla{q} = 0 & $ \text{in } \,\Omega \times [0, T_0]$ , \label{construction-equation-1} \\
     q|_{t=0} = q_0 & $ \text{in } \,  \Omega$,  \label{construction-equation-2}\\
     \Delta_x {\psi}(x, t) = {q}(x, t) & $ \text{in } \,\Omega \times [0, T_0]$ \label{pf-poisson-psi},  \label{construction-equation-3}\\
     {\psi}(x, t) = 0 & $ \text{on } \, \Gamma \times [0, T_0]$ \label{pf-poisson-psi-2},  \label{construction-equation-4}\\
      {\tilde{\bm{u}}}(x, t) + A\tilde{\bm{u}}(x, t) = \nabla^{\perp}{\psi}(x, t) &  $\text{in } \Omega \times [0, T_0]$,   \label{construction-equation-5}\\
       \tilde{\bm{u}}(x, t) = 0 & $\text{on } \Gamma \times [0, T_0]$.   \label{construction-equation-6}
\end{numcases}
Indeed, from Lemma \ref{lemma-transport}, we know that the transport equations (\ref{construction-equation-1})--(\ref{construction-equation-2}) admits a unique weak solution ${q} \in C([0, T_0]; L^2(\Omega) \cap L^1(\Omega))$ with the following estimate
\begin{equation}\label{pf-transport}
\sup_{t\in [0, T_0]}\| ({q(t)} \|_{L^2(\Omega)} + \| {q(t)}\|_{L^1(\Omega)}) \leqslant (\| q_0 \|_{L^2(\Omega)} + \|q_0\|_{L^1(\Omega)}).
\end{equation}
Moreover, concerning Remark \ref{remark-transport-equation}, we observe that the diameter of the support of $q(t)$ obeys
\begin{equation}\label{pf-transport-1}
\delta(\mathrm{supp} \,q(t))< R_0 + C\int_0^t \|\bm{u}(\cdot, s)\|_{\bm{H}^2(\Omega)}\dif s,
\end{equation}
 for $t \in [0, T_0]$. In the light of Lemma \ref{lemma-poisson}, we know that the Poisson equation (\ref{construction-equation-3})$-$(\ref{construction-equation-4})
has a unique solution ${\psi} \in C\left([0, T_0]; \dot{H}\right)$ such that $\nabla {\psi}$ belongs to $C([0, T_0]; \bm{H}^1(\Omega))$ and satisfies the following property
\begin{equation}\label{pf-poisson}
\| \nabla {\psi}(t) \|_{\bm{H}^1(\Omega)} \leqslant C\delta(\supp\, q(t)) (\|q(t)\|_{L^2(\Omega)} + \| {q}(t) \|_{L^1(\Omega)}), \ \forall t \in [0, T_0].
\end{equation}
Furthermore, By Lemma \ref{lemma-stokes}, we observe that the Stokes equation (\ref{construction-equation-5})$-$(\ref{construction-equation-6}) has a unique solution $\bm{\tilde{u}} \in C([0, T_0]; \bm{H}^3(\Omega) \cap V)$ such that for all $t \in [0, T_0]$, the following estimate holds
\begin{equation}\label{pf-stokes}
\begin{aligned}
\|\bm{\tilde{u}}(t)\|_{\bm{H}^3(\Omega)} &\leqslant C \|\nabla^\bot {\psi}(t)\|_{\bm{H}^1(\Omega)}.
\end{aligned}
\end{equation}
The contraction map is then defined by $\cal{F}[\bm{u}] := \bm{\tilde{u}}$.

$\bm{Step \ 2}$  We shall now determine the parameters $T_0$ and $M$ to ensure that $\bm{\tilde{u}} \in D(\mathcal{F})$. Collecting $(\ref{pf-transport-1})$, $(\ref{pf-poisson})$ and $(\ref{pf-stokes})$, it follows that:
\begin{align}\label{pf-estimate-tilde-u}
\|\bm{\tilde{u}}(t)\|_{\bm{H}^3(\Omega)} &\leqslant C(R_0 + \int_0^t \|\bm{u}(\cdot, s)\|_{\bm{H}^2(\Omega)} \dif s) (\|q_0\|_{L^1(\Omega)} + \|q_0\|_{L^2(\Omega)})
\end{align}
for $t \in [0, T_0]$. Setting $M := 2 \max\left(C(\| q_0 \|_{L^2(\Omega)} + \|q_0\|_{L^1(\Omega)}), \|\bm{u}_0\|_{\bm{H}^3(\Omega)}\right)$ and $T_0 := \frac{R_0}{M}$, taking account of ($\ref{pf-estimate-tilde-u}$) it follows that
\begin{equation}\label{pf-tilde-u-estimate}
 \sup_{t\in[0, T_0]}\|\bm{\tilde{u}}(t)\|_{\bm{H}^3(\Omega)} \leqslant M.
\end{equation}
To ensure that $\bm{\tilde{u}} \in D(\cal{F})$, it remains to check that $\bm{\tilde{u}}|_{t=0} = \bm{u}_0$. From equations (\ref{construction-equation-1})$-$(\ref{construction-equation-6}), we deduce that
\begin{equation*}
\begin{aligned}
\nabla ^\perp \cdot (\bm{\tilde{u}} + A\bm{\tilde{u}})|_{t=0} = \Delta {\psi}|_{t=0} = q_0 = \nabla ^\perp\cdot(\bm{u}_0 + A\bm{u}_0),
\end{aligned}
\end{equation*}
  Since both $\bm{u}$ and $\bm{\tilde{u}}$ are  divergence-free vectors, we at once obtain
\begin{equation*}
\begin{aligned}
\nabla \cdot (\bm{\tilde{u}} + A\bm{\tilde{u}})|_{t=0} = \nabla \cdot (\bm{u}_0 + A\bm{u}_0) = 0.
\end{aligned}
\end{equation*}
Collecting the above two equalities, in view of Lemma $\ref{lemma-harmonic}$, we thus have
\begin{equation*}
\begin{aligned}
(\bm{\tilde{u}}|_{t=0} -\bm{u}_0) + A(\bm{\tilde{u}}|_{t=0} - \bm{u}_0) = 0,
\end{aligned}
\end{equation*}
which, in turn, secures $\bm{\tilde{u}}|_{t=0} = \bm{u}_0$ by Lemma \ref{lemma-stokes}.

$\bm{Step \ 3}$  We have to prove that the map $\mathcal{F}$ is a contraction. First of all, we assert that ($\bm{u}$, $\bm{\tilde{u}}$) satisfies the following equations
\begin{numcases}{}
    \partial_t \bm{\tilde{v}} + \bm{u} \cdot {\nabla} \bm{\tilde{v}} - (\nabla \bm{\tilde{v}})^t \bm{u}  + \nabla \tilde{p} = 0  &  $\text{in} \ \Omega \times (0, T_0] $ \label{pf-euler-alpha-1}\\
    \text{div} \ \bm{\tilde{u}} = 0, &  $\text{in} \ \Omega \times [0, T_0] $ \label{pf-euler-alpha-2}\\
    \bm{\tilde{u}} = 0,  & $\text{in} \ \Gamma \times [0, T_0] \label{pf-euler-alpha-3}$\\
    \bm{\tilde{u}}|_{t=0} = \bm{u}_0, & $\text{in} \ \Omega $\label{pf-euler-alpha-4}\\
    \bm{\tilde{u}} (x, t) \to 0, & $\forall t \in [0, T_0], |x| \to \infty $ \label{pf-euler-alpha-5}
\end{numcases}
where $\bm{\tilde{v}} = \bm{\tilde{u}} - \Delta \bm{\tilde{u}}$. We only need to check ($\ref{pf-euler-alpha-1}$). To end this, let
\begin{equation*}
\begin{aligned}
\nonumber \bm{Q}(t) :=  \bm{\tilde{v}}(t) -   \bm{\tilde{v}}(0) + \int_0^t \left[\bm{u}(s) \cdot {\nabla} \bm{\tilde{v}}(s)  - (\nabla \bm{\tilde{v}})^t(s) \bm{u}(s) \right] \dif s \ \ \forall t \in [0, T_0],
\end{aligned}{}
\end{equation*}
 Observe that $\bm{u}, \bm{\tilde{u}} \in C([0, T_0]; \bm{H}^3(\Omega)\cap V)$, it follows $\bm{Q}(t) \in C([0, T_0]; (L^2(\Omega))^2)$. Taking account of \eqref{construction-equation-1}, we also know that
\begin{equation*}
\begin{aligned}
\nonumber \nabla_x^\perp\cdot \bm{Q}(x, t) &=\int_0^t \left[ \partial_t {q}(x, s) + \bm{u} \cdot \nabla_x{{q}(x, s)}\right] \dif s \\
&= 0.
\end{aligned}
\end{equation*}
From Lemma \ref{lemma-harmonic}, we deduce that there exists $\nabla \tilde{P} \in C([0, T_0];(L^2(\Omega))^2)$ such that
\begin{equation*}
\begin{aligned}
 \nonumber \bm{Q}(t) = \nabla \tilde{P}(t)
\end{aligned}
\end{equation*}
for $t \in [0, T_0]$. Differentiate the above equality with respect to t and denote $\partial_t \tilde{P}(t)$ by $\tilde{p}(t)$, we then arrive at ($\ref{pf-euler-alpha-1}$) in distribution sense.
We now focus on the energy estimates for these equations \eqref{pf-euler-alpha-1}--\eqref{pf-euler-alpha-5} to show $\cal{F}$ is a contraction. Assume that $\bm{u^1}, \bm{u^2} \in C([0, T_0]; V\cap \bm{H}^3(\Omega))$, we set
\begin{equation*}
\begin{aligned}
\bm{\tilde{u}}^1 = \cal{F}[{\bm{u^1}}],\\
 \bm{\tilde{u}}^2 = \cal{F}[\bm{u}^2],\\
 \bm{\tilde{v}}^1 = \bm{\tilde{v}}^1 -\Delta\bm{\tilde{u}}^1\\ \bm{\tilde{v}}^2 = \bm{\tilde{u}}^1 - \Delta\bm{\tilde{u}}^2,\\
W = \bm{u^1} - \bm{u^2},\\
S = \bm{\tilde{u}}^1 - \bm{\tilde{u}}^2,
\end{aligned}
\end{equation*}
then $\bm{\tilde{v}}^1 - \bm{\tilde{v}}^2 = S - \Delta S$. Subtracting  the equation ($\ref{pf-euler-alpha-1})$ for $\bm{\tilde{u}}^2$ from the one for $\bm{\tilde{u}}^1$, it follows
\begin{equation*}
\begin{aligned}
\partial_t(S-\Delta S) + \bm{u^1} \cdot \nabla \bm{\tilde{v}}^1 - \bm{u^2} \cdot \nabla \bm{\tilde{v}}^2 + \sum_{j=1}^2 \bm{u^1}_j \cdot \nabla \bm{\tilde{v}}^1_j - \sum_{j=1}^2 \bm{u^2}_j \cdot \nabla \bm{\tilde{v}}^2_j + \nabla \tilde{p}^1 - \nabla \tilde{p}^2 = 0.
\end{aligned}
\end{equation*}
Multiplying the above equation by $S$ and integrating over $\Omega \times [0, t)$ for any $t \in (0, T_0]$, we obtain, after integrate by parts,
\begin{equation}\label{pf-K}
\begin{aligned}
\frac{1}{2}\left(\|S(t)\|_{L^2}^2 + \|\nabla S(t)\|_{L^2}^2\right) &= \int_0^t\int_{\Omega} S \cdot \left[ (W\cdot \nabla) \bm{\tilde{v}}^1 + (\bm{u^2}\cdot \nabla) (S - \Delta S)\right] \dif x \dif s\\
&\quad+\int_0^t\int_{\Omega} S \cdot \left[ \sum_{j=1}^2 \bm{u}_j^1\nabla (S_j - \Delta S_j) + \sum_{j=1}^2 W_j \nabla \bm{\tilde{v}}_j^2\right] \dif x \dif s\\
&=:K_1 + K_2,
\end{aligned}
\end{equation}
where we have used the property that $\tilde{\bm{u}}^2|_{t = 0} = \tilde{\bm{u}}^1_{t = 0} = \bm{u}_0$.  As usual, noting that $(S, \bm{u^2} \cdot \nabla S) = 0$ and integrating by parts, we deduce:
\begin{align*}
\nonumber K_1 &= \int_0^t\int_{\Omega}S \cdot \left[ (W\cdot \nabla) \bm{\tilde{v}}^1 + (\bm{u^2}\cdot \nabla) (S - \Delta S)\right] \dif x \dif s\\
\nonumber &= \int_0^t\int_{\Omega}S \cdot \left[ (W\cdot \nabla) \bm{\tilde{v}}^1\right] \dif x \dif s -\int_0^t\int_{\Omega}S \cdot \left[ (\bm{u^2}\cdot \nabla) (\Delta S)\right] \dif x \dif s \\
\nonumber &= \int_0^t\int_{\Omega}S \cdot \left[ (W\cdot \nabla) \bm{\tilde{v}}^1\right] \dif x \dif s +\int_0^t\int_{\Omega}\Delta S \cdot \left[ (\bm{u^2}\cdot \nabla) S\right] \dif x \dif s\\
\nonumber&=  \int_0^t\int_{\Omega}S \cdot \left[ (W\cdot \nabla) \bm{\tilde{v}}^1 \right] \dif x \dif s  - \int_0^t\int_{\Omega}\sum_{k, l} \left[ \partial_l u_k^2 \partial_k S + u_k^2\partial_k\partial_l S\right] \partial_l S \dif x \dif s\\
\nonumber&= \int_0^t\int_{\Omega}S \cdot \left[ (W\cdot \nabla) \bm{\tilde{v}}^1 \right] \dif x \dif s  - \int_0^t\int_{\Omega}\sum_{k, l} \left[ \partial_l u_k^2 \partial_k S\right] \partial_l S \dif x \dif s,
\end{align*}
  then taking account of ($\ref{pf-tilde-u-estimate}$) and using H\"older inqueality and Gagliardo-Nirenberg inequality, it follows
\begin{align}
\nonumber |K_1| &\leqslant \int_0^t \left[\|S\|_{L^4} \|W\|_{L^4} \|\nabla \bm{\tilde{v}}^1\|_{L^2} +  \|\nabla \bm{u^2}\|_{L^\infty}\|\nabla S\|_{L^2}^2\right] \dif s\\
%\nonumber&\leqslant C (\|S\|_{L^2}^{\frac{1}{2}}\|\nabla S\|_{L^2}^{\frac{1}{2}}\|W\|_{L^2}^{\frac{1}{2}}\|\nabla W\|_{L^2}^{\frac{1}{2}} \|\bm{\tilde{v}}^1\|_{\bm{H}^1(\Omega)} + \|\bm{u^2}\|_{\bm{H}^3}\|S\|_{\bm{H}^1(\Omega)})\\
%\nonumber&\leqslant C(\|\bm{\tilde{v}}^1\|_{\bm{H}^1(\Omega)} \|S\|_{\bm{H}^1(\Omega)}\|W\|_{\bm{H}^1(\Omega)} + \|\bm{u^2}\|_{\bm{H}^3}\|S\|_{\bm{H}^1(\Omega)}^2)\\
&\leqslant CM \int_0^t\left[\|S\|_{\bm{H}^1(\Omega)}^2 + \|W\|_{\bm{H}^1(\Omega)}^2\right] \dif s. \label{pf-K_1}
\end{align}
 We now estimate the second term $K_2$, which is similar to $K_1$. Indeed, by integration by parts, $K_2$ can be rewrited as:
\begin{align*}
\nonumber K_2 &= \int_0^t\int_{\Omega} S \cdot \left[ \bm{u}_j^1\nabla (S_j -  \Delta S_j) +  W_j \nabla \bm{\tilde{v}}_j^2\right] \dif x \dif s \\
\nonumber &= \int_0^t\int_{\Omega} \left[S_i \bm{u}_j^1\partial_i S_j -  S_i \bm{u}_j^1\partial_i(\Delta S_j) +  S_i W_j \partial_i \bm{\tilde{v}}_j^2 \right]\dif x \dif s\\
\nonumber &= \int_0^t\int_{\Omega} \left[S_i \bm{u}_j^1\partial_i S_j  +  S_i \partial_i \bm{u}_j^1 \Delta S_j +  S_i W_j \partial_i \bm{\tilde{v}}_j^2 \right] \dif x \dif s \\
\nonumber &= \int_0^t\int_{\Omega} \left[S_i\bm{u}_j^1 \partial_i S_j - \partial_k \partial_l \bm{u^1}_jS_l\partial_k S_j
- \partial_l \bm{u^1}_j\partial_kS _l\partial_k S_j +  \partial_i \bm{\tilde{v}}^2_j S_i W_j\right] \dif x \dif s,
\end{align*}
again  using inequality \eqref{pf-tilde-u-estimate}, H\"older inqueality and Gagliardo-Nirenberg inequality, it is easy to deduce that
\begin{align}
\nonumber |K_2| &\leqslant \int_0^t \left[\|S\|_{L^2}\|\bm{u^1}\|_{L^\infty}\|\nabla S\|_{L^2} + \|\nabla^2\bm{u^1}\|_{L^4}\|S\|_{L^4}\|\nabla S\|_{L^2} \right] \dif s \\
\nonumber &\quad + \int_0^t \left[\|\nabla \bm{u^1}\|_{L^\infty}\|\nabla S\|_{L^2}^2  + \|S\|_{L^4}\|W\|_{L^4}\|\nabla \bm{\tilde{v}}^2\|_{L^2}\right] \dif s\\
&\leqslant CM \int_0^t \left[\|S\|_{\bm{H}^1(\Omega)}^2 + \|W\|_{\bm{H}^1(\Omega)}^2\right] \dif s. \label{pf-K_2}
\end{align}
In view of (\ref{pf-K}), (\ref{pf-K_1}) and (\ref{pf-K_2}), it shows that for all $t \in [0, T_0]$
\begin{equation*}
\begin{aligned}
\frac{1}{2} \|S(t)\|_{\bm{H}^1(\Omega)}^2 \leqslant CM \int_0^t \left[\|S\|_{\bm{H}^1(\Omega)}^2 + \|W\|_{\bm{H}^1(\Omega)}^2\right] \dif s
\end{aligned}
\end{equation*}
where $M$ is defined in step 2 and $C$ only depends on $\Omega$. Thanks to Gr\"onwall's inequality, we have
\begin{equation*}
\begin{aligned}
\|S\|_{\bm{H}^1(\Omega)}^2(t) \leqslant CM\int_0^t\|W\|_{\bm{H}^1(\Omega)}^2(s) \exp\{CM(t-s) \}\dif s.
\end{aligned}
\end{equation*}
For arbitrary $h \in (0, T_0]$, the above inequality implies
\begin{equation}
\begin{aligned}\label{pf-estimate-s}
\sup_{t\in[0, h]} \|S\|_{\bm{H}^1(\Omega)}^2(t) \leqslant ({e^{CM T} - 1})\sup_{t\in[0, h]}\|W\|_{\bm{H}^1(\Omega)}^2(t).
\end{aligned}
\end{equation}
Consequently, if choose $h \in (0, T_0] $ small enough, so that $e^{CM T} - 1 < 1$, then we have that $\cal{F}$ is a contraction mapping with respect to the $\bm{H}^1(\Omega)$ norm, for short time interval $[0, h]$.

$\bm{Step \ 4}$  Local existence. By Banach Fixed Point Theorem , we could conclude there exists a unique fixed point $\bm{u}\in C([0, h]; V)$. This fixed point is also the limit of the fixed point iteration, with $\bm{u}^0 := \bm{u}_0$ and $\bm{u}^n := \cal{F}[\bm{u}^{n-1}]$. From $(\ref{pf-tilde-u-estimate})$, we know that $\bm{u}^n$ is uniformly bounded in $\bm{H}^3(\Omega)$, thanks to  Banach-Alaoglu theorem we have that there exists a subsequence $\bm{u}^{n_k}$ converges, weak-star to $\bm{u}$ in  $L^\infty([0, h]; \bm{H}^3(\Omega))$. It's easy to see that $\bm{u} \in C([0, T]; \bm{H}^3(\Omega))$. Indeed, since $\cal{F}[\bm{u}] \in C([0, T]; \bm{H}^3(\Omega)\cap V)$ and $\bm{u}$ is the fixed point, we have $\bm{u} := \cal{F}[\bm{u}] \in C([0, T]; \bm{H}^3(\Omega)\cap V)$. %Suppose $\delta < T_0$, then let $\bar{\bm{u}}_0 := \bm{u}|_{t=\delta}$ and repeat the above steps with $\bm{u}_0$ replaced by $\bar{\bm{u}}_0$, we could obtain the same estimates $(\ref{pf-tilde-u-estimate})$, for which we get the same inequality $(\ref{pf-estimate-s})$  to extend the solution to time interval $[0, 2\delta]$. Repeating this process, we could extend to $[0, T_0]$.

$\bm{Step \ 5}$  Extending the solution to any time interval. Since $\bm{u}$ is the limit of $\bm{u}^{n_k}$, it follows from $(\ref{pf-estimate-tilde-u})$ that
\begin{align*}
\nonumber \|\bm{u}(\cdot, t)\|_{\bm{H}^3} &\leqslant C(R_0 + \int_0^t\|\bm{u}(\cdot, s)\|_{\bm{H}^2(\Omega)}\dif s) (\|q_0\|_{L^1} + \|q_0\|_{L^2}) \\
&\leqslant C(R_0 + \int_0^t\|\bm{u}(\cdot, s)\|_{\bm{H}^3(\Omega)}\dif s)
\end{align*}
for $t \in [0, h]$ and $C$ is independent of time $t$. Using Gr\"onwall's inequality, we obtain:
\begin{align*}
\|\bm{u}\|_{L^\infty([0, h]; \bm{H}^3(\Omega))} \leqslant CR_0 \exp (Ch),
\end{align*}
 which implies the solution can be extended to any time interval.
\end{proof}

\begin{remark} By combing $\eqref{high-transport-estimate}$ and the arguments in the proof of Proposition $\ref{proposition-1}$,  one is able to obtain the following estimate provided that  $\bm{u}_0 \in H^s(\Omega)$ for $s > 3$,
\begin{align}
\|\bm{u}\|_{L^\infty([0, T];\bm{H}^s(\Omega))} \leq C\|\bm{u}_0\|_{\bm{H}^s(\Omega)}
\end{align}
where C depends on T and $\bm{H}^3(\Omega)$-norm of $\bm{u}_0$.
\end{remark}
With the above technical lemmas and proposition, we finally can obtain the main result in this paper.
\begin{main}\label{main-theorem}
Assume that the initial \textit{filtered} velocity $\bm{u}_0 \in \bm{H}^s(\Omega) \cap V$($s \geq 3$), then for arbitrary $T>0$, the equations $(\ref{euler-alpha-1r})$--$(\ref{euler-alpha-5r})$ admit a unique weak solution $\bm{u} \in L^\infty(0, T; \bm{H}^s(\Omega)) \cap C([0, T]; V)$ in the following sense, for any $\phi \in C^\infty([0, T]; \mathcal{D})$, the indentity
\begin{equation}\label{energy-formula-main-theorem}
\begin{aligned}
 (\bm{u}(t), \phi)_{L^2(\Omega)} + (\nabla \bm{u}(t), \nabla \phi) - (\bm{u}_0, \phi) - (\nabla \bm{u}_0, \nabla \phi)_{L^2(\Omega)} \\
= \int_0^t(\bm{v}, \phi_t)_{L^2(\Omega)} - \int_0^t (\bm{u}\cdot \nabla \bm{v} + (\nabla \bm{u})^t \cdot \bm{v}, \phi)_{L^2(\Omega)}
\end{aligned}
\end{equation}
holds for $t\in[0, T]$, where $\bm{v} = \bm{u} - \Delta\bm{u}$. Moreover,
\begin{align*}
\|\bm{u}(t)\|_{H^s(\Omega)} \leqslant C \|\bm{u}_0\|_{H^s(\Omega)},
\end{align*}
where C depends on $\Omega$, $T$ and $\|\bm{u}_0\|_{\bm{H}^3(\Omega)}$ for $s > 3$, while depends only on $\Omega$  for $s = 3$.
\end{main}
\begin{proof}
 Let $T > 0$ be arbitrary. Assume that $\bm{u}_0 \in V \cap \bm{H}^s(\Omega)$, $s \geq 3$ and choose $\lbrace \bm{u}^n_0\rbrace$ be the approximate sequence constructed in Lemma \ref{approximate-lemma}. For any fixed $n \in \mathbb{N}$, we know from Proposition \ref{proposition-1} that there exists a unique solution $\bm{u}^n$ to equations \eqref{euler-alpha-1r}-\eqref{euler-alpha-5r} with initial value $\bm{u}^n_0$.

  It is crucial to prove that the limit of $\bm{u}^n$ in the suitable space is  the solution $\bm{u}$ of the formula \eqref{energy-formula} with the initial data $\bm{u}_0$. Fortunately, we observe that it is able to obtain the uniform bound  for the lower order partial derivatives of $\bm{u}^n$ from equation ($\ref{euler-alpha-1r}$) and the one for the high order derivatives from equations \eqref{construction-equation-1}--\eqref{construction-equation-6}, respectively.

 Beginning to the estimates of low order derivatives, it is easy to obtain this bound from energy estimates for the equation \eqref{euler-alpha-1r} as follows,
\begin{equation*}
\begin{aligned}
\frac{1}{2}\frac{\dif}{\dif t} (\|\bm{u}^n\|_{L^2(\Omega)}^2 + \|\nabla \bm{u}^n\|_{L^2(\Omega)}^2) &=
\int_{\Omega} \bm{u}^n_j \partial_j(\bm{u}^n_i -  \Delta \bm{u}^n_i) \bm{u}^n_i + (\bm{u}^n_j -  \Delta \bm{u}^n_j) \partial_i \bm{u}^n_j \bm{u}^n_i\\
&=\int_{\Omega} \bm{u}^n_j \partial_j(\frac{1}{2}\bm{u}^n_i)^2  -  \int_{\Omega}  \bm{u}^n_j  \partial_j \Delta \bm{u}^n_i \bm{u}^n_i \\
&\quad+ \int_{\Omega} (\bm{u}^n_j) \partial_i \bm{u}^n_j \bm{u}^n_i -  \int_{\Omega} \Delta \bm{u}^n_j \partial_i \bm{u}^n_j \bm{u}^n_i\\
&=-  \int_{\Omega}  \bm{u}^n_j  \partial_j \Delta \bm{u}^n_i \bm{u}^n_i + \Delta \bm{u}^n_j \partial_i \bm{u}^n_j \bm{u}^n_i\\
&=0,
\end{aligned}
\end{equation*}
where we used the condition $\nabla \cdot \bm{u}^n \equiv 0$. It follows that
\begin{equation}\label{pf-low-esimtate-2}
\begin{aligned}
\sup_{t \in [0, \infty)} \|\bm{u}^n(\cdot, t)\|_{\bm{H}^1(\Omega)} \leqslant  \|\bm{u}^n_0\|_{\bm{H}^1(\Omega)} \leqslant C \|\bm{u}_0\|_{\bm{H}^1(\Omega)}.
\end{aligned}
\end{equation}

It is subtle to obtain the uniform estimates of the high order derivatives. As in the proof in Proposition \ref{proposition-1}, the solution $\bm{u}^n$ satisfies $\mathcal{F}[\bm{u}^n] = \bm{u}^n$, by $(\ref{construction-equation-1})$--$(\ref{construction-equation-6})$ we know that there exists the triple $(\bm{u}^n, q^n, \psi^n)$  obeys the following system
\begin{numcases}{}
    \nonumber \partial_t q^n + \bm{u}^n \cdot \nabla{q^n} = 0 & $ \text{in } \,\Omega \times [0, T]$ , \\
    \nonumber q^n|_{t=0} = q_0^n & $ \text{in } \,  \Omega$, \\
    \nonumber \Delta_x {\psi^n}(x, t) = {q^n}(x, t) & $ \text{in } \,\Omega \times [0, T]$ \\
    \nonumber {\psi^n}(x, t) = 0 & $ \text{on } \, \Gamma \times [0, T]$ \\
    \nonumber {{\bm{u}}}^n(x, t) -\Delta{\bm{u}}^n(x, t)+\nabla p^n = \nabla^{\perp}{\psi^n}(x, t) &  $\text{in } \Omega \times [0, T]$,   \\
    \nonumber {\bm{u}}^n(x, t) = 0 & $\text{on } \Gamma \times [0, T]$.
\end{numcases}
 Since $(\bm{u}^n, \nabla^\perp \psi^n)$ satisfies the stationary Stokes equations, with the aid of proposition 6 in \cite{giga2018handbook}, we obtain
\begin{equation*}
\|D^{k+3}\bm{u}^n(\cdot, t)\|_{L^2(\Omega)} \leqslant C(\|\nabla^\perp\psi^n(\cdot, t)\|_{H^{k+1}(\Omega)} + \|\bm{u}^n(\cdot, t)\|_{L^2(\Omega)})
\end{equation*}
for $t \in [0, T]$ and $k \geq 0$. Collecting the inequality ($\ref{esimate-du-m}$) and the above inequality, it follows
\begin{align*}
\|D^{k+3}\bm{u}^n(\cdot, t)\|_{L^2(\Omega)} \leqslant C(\|q^n(\cdot, t)\|_{H^k(\Omega)} + \|\nabla\psi^n\|_{L^2(\Omega))} + \|\bm{u}^n(\cdot, t)\|_{L^2(\Omega)}).
\end{align*}
Noting that  $\|\nabla^\perp\psi^n\|_{L^2(\Omega)} \leqslant \|\bm{u}^n\|_{H^2(\Omega)}$, by Sobolev interpolation inequality, we then arrive at
\begin{equation*}
\begin{aligned}
\|D^{k+3}\bm{u}^n(\cdot, t)\|_{L^2(\Omega)} &\leqslant C(\|q^n(\cdot, t)\|_{H^k(\Omega)} + \|\bm{u}^n(\cdot, t)\|_{L^2(\Omega)}),
\end{aligned}
\end{equation*}
for $t \in [0, T]$ and $k \geq 0$. Let's consider the special case that $k = 0$, By the above inequality, ($\ref{transport-3}$) and ($\ref{pf-low-esimtate-2}$), we can see that
\begin{equation*}
\begin{aligned}
\|\bm{u}^n\|_{L^\infty([0, T];\bm{H}^3(\Omega))} &\leqslant C(\|q^n_0\|_{L^2(\Omega)} + \|\bm{u}^n_0\|_{L^2(\Omega)})\leqslant C\|\bm{u}^n_0\|_{\bm{H}^3(\Omega)}\leqslant C\|\bm{u}_0\|_{\bm{H}^3(\Omega)}.
\end{aligned}
\end{equation*}
Taking into account the above two inequalities and (\ref{high-transport-estimate}), we finally deduce by induction that
\begin{equation*}
\begin{aligned}
\|\bm{u}^n\|_{L^\infty([0, T]; \bm{H}^s(\Omega))} &\leqslant C\|\bm{u}_0\|_{H^s(\Omega)},
\end{aligned}
\end{equation*}
where $C$ depends on $T$ and $\bm{H}^3(\Omega)$-norm of $\bm{u}_0$ for $s \geq 4$. By virtue of Banach-Alaoglu Theorem, we find that there exists a subsequence $\bm{u}^{n_k}$ converges weak-star to some $\bm{u}$ in $L^\infty\left([0, T]; \bm{H}^s(\Omega)\right)$.

Furthermore, we assert that $\bm{u}^n$ converges to $\bm{u}$ in $C([0, T]; \bm{H}^1(\Omega))$ strongly. Firstly, subtracting the equation ($\ref{euler-alpha-1r}$) for $\bm{u}^m$ from the one for $\bm{u}^n$, respectively, then multiplying by $(\bm{u}^n - \bm{u}^m)$ and integrating over $\Omega$, we have that
\begin{equation*}
\begin{aligned}
\frac{1}{2} \frac{\dif}{\dif t}\|\bm{u}^n - \bm{u}^m)\|^2_{\bm{H}^1(\Omega)}&= \int_{\Omega} [(\bm{u}^n - \bm{u}^m) \cdot \nabla \bm{v}^n] \cdot (\bm{u}^n - \bm{u}^m) \\
&\quad + \int_{\Omega} [\bm{u}^m \cdot \nabla (\bm{v}^n - \bm{v}^m)] \cdot (\bm{u}^n - \bm{u}^m)\\
&\quad + \int_{\Omega} [(\nabla \bm{u}^n - \nabla \bm{u}^m)^t \cdot \bm{v}^n] \cdot (\bm{u}^n - \bm{u}^m) \\
&\quad + \int_{\Omega} [(\nabla \bm{u}^m)^t \cdot (\bm{v}^n - \bm{v}^m)] \cdot (\bm{u}^n - \bm{u}^m)\\
&=: M_1 + M_2 + M_3 + M_4.
\end{aligned}
\end{equation*}
It is easy to observed that $M_1 + M_3 = 0$, clearly, integrating by parts, we find
\begin{equation*}
\begin{aligned}
\nonumber M_3 &= \int_{\Omega} [(\nabla \bm{u}^n - \nabla \bm{u}^m)^t \cdot \bm{v}^n] \cdot (\bm{u}^n - \bm{u}^m) \\
\nonumber &= \int_{\Omega} \partial_i (\bm{u}^n_j - \bm{u}^m_j) \bm{v}^n_j(\bm{u}^n_i - \bm{u}^m_i) \\
%\nonumber &= -\int_{\Omega} (\bm{u}^n_j - \bm{u}^m_j)\partial_i \bm{v}^n_j(\bm{u}^n_i - \bm{u}^m_i) \\
\nonumber &= -M_1.
\end{aligned}
\end{equation*}
 Using integration by parts, we obtain  more favourable formula for $M_2$
\begin{equation*}
\begin{aligned}
\nonumber M_2 & = \int_{\Omega} [\bm{u}^m \cdot \nabla (\bm{v}^n - \bm{v}^m)] \cdot (\bm{u}^n - \bm{u}^m)\\
\nonumber&=-\int_{\Omega} [\bm{u}^m \cdot \nabla (\bm{u}^n - \bm{u}^m)] \cdot (\bm{v}^n - \bm{v}^m)\\
\nonumber &=-\int_{\Omega} [\bm{u}^m \cdot \nabla (\bm{u}^n - \bm{u}^m)] \cdot (\bm{u}^n - \bm{u}^m) +\int_{\Omega} [\bm{u}^m \cdot \nabla (\bm{u}^n - \bm{u}^m)] \cdot (\Delta\bm{u}^n - \Delta\bm{u}^m)\\
\nonumber &=-\int_{\Omega} [\bm{u}^m \cdot \nabla (\bm{u}^n - \bm{u}^m)] \cdot (\bm{u}^n - \bm{u}^m)- \int_{\Omega} [\partial_i\bm{u}^m_j \cdot \partial_j (\bm{u}^n_k - \bm{u}^m_k)] \cdot \partial_i(\bm{u}^n_k - \bm{u}^m_k),\\
\end{aligned}
\end{equation*}
with the aid of H\"older inequality and Sobolev embedding theorem, we arrive at
\begin{equation*}
\begin{aligned}
M_2 &\leqslant \|\bm{u}^m\|_{L^\infty} \|\bm{u}^n - \bm{u}^m\|_{\bm{H}^1(\Omega)}^2 + \|\nabla \bm{u}^m\|_{L^\infty} \|\bm{u}^n - \bm{u}^m\|_{\bm{H}^1(\Omega)}^2\\
&\leqslant C\|\bm{u}^m\|_{\bm{H}^3(\Omega)} \|\bm{u}^n - \bm{u}^m\|_{\bm{H}^1(\Omega)}^2.
\end{aligned}
\end{equation*}

It remains to check the estimate of the term $M_4$. Similarly, using integration by parts, $M_4$ can be rewritten as follows
\begin{equation*}
\begin{aligned}
\nonumber M_4 &= \int_{\Omega} [(\nabla \bm{u}^m)^t \cdot (\bm{v}^n - \bm{v}^m)] \cdot (\bm{u}^n - \bm{u}^m)\\
\nonumber &= \int_{\Omega} \partial_i \bm{u}^m_j (\bm{v}^n_j - \bm{v}^m_j) (\bm{u}^n_i - \bm{u}^m_i)\\
\nonumber &= \int_{\Omega} \partial_i \bm{u}^m_j (\bm{u}^n_j - \bm{u}^m_j) (\bm{u}^n_i - \bm{u}^m_i)- \int_{\Omega} \partial_i \bm{u}^m_j (\Delta\bm{u}^n_j - \Delta\bm{u}^m_j) (\bm{u}^n_i - \bm{u}^m_i)\\
\nonumber &= \int_{\Omega} \partial_i \bm{u}^m_j (\bm{u}^n_j - \bm{u}^m_j) (\bm{u}^n_i - \bm{u}^m_i) + \int_{\Omega} \partial_k\partial_i \bm{u}^m_j \partial_k(\bm{u}^n_j - \bm{u}^m_j) (\bm{u}^n_i - \bm{u}^m_i)\\
\nonumber &\quad+ \int_{\Omega} \partial_i \bm{u}^m_j \partial_k(\bm{u}^n_j - \bm{u}^m_j) \partial_k(\bm{u}^n_i - \bm{u}^m_i),
\end{aligned}
\end{equation*}
then applying H\"older inequality and Gagliardo-Nirenberg  inequality, we deduce, in particular
\begin{equation*}
\begin{aligned}
M_4 &\leqslant \|\nabla \bm{u}^m\|_{L^\infty} \|\bm{u}^n -\bm{u}^m\|_{L^2}^2 + \|D^2\bm{u}^m\|_{L^4}\|\nabla\bm{u}^n - \nabla\bm{u}^m\|_{L^2}\|\bm{u}^n - \bm{u}^m\|_{L^4} \\
&\quad+ \|\nabla \bm{u}^m\|_{L^\infty} \|\nabla \bm{u}^n -\nabla\bm{u}^m\|_{L^2}^2 \\
&\leqslant C\|\bm{u}^m\|_{\bm{H}^3(\Omega)} \|\bm{u}^n -\bm{u}^m\|_{\bm{H}^1(\Omega)}^2.
\end{aligned}
\end{equation*}
The previous estimates about $M_1, M_2, M_3$ and $M_4$ then imply that
\begin{equation*}
\begin{aligned}
\frac{1}{2} \frac{\dif}{\dif t}\|\bm{u}^n - \bm{u}^m)\|^2_{\bm{H}^1(\Omega)} \leqslant C\|\bm{u}^n -\bm{u}^m\|_{\bm{H}^1(\Omega)}^2,
\end{aligned}
\end{equation*}
 thanks to Gr\"onwall's inequality, we obtain
\begin{equation*}
\begin{aligned}
\sup_{t\in[0, T]} \|\bm{u}^n - \bm{u}^m\|_{\bm{H}^1(\Omega)}(t) \leqslant  e^{CT} \|\bm{u}^n_0 - \bm{u}^m_0\|_{\bm{H}^1(\Omega)}.
\end{aligned}
\end{equation*}
Consequently, we can conclude that $\bm{u}^n$ converges strongly in $C([0, T]; \bm{H}^1(\Omega))$ to $\bm{u}$ with $\bm{u}|_{t=0} = \bm{u}_0$.

Since $\bm{u}^n$ converges weak-star (respectively, strongly) to $\bm{u}$ in $L^\infty([0, T]; \bm{H}^s(\Omega)$)(respectively, $C([0, T]; \bm{H}^1(\Omega))$), and then it is easy to check that $\bm{u}$ is a solenoidal vector,  we thus know that $\bm{u}$ satisfies the equations ($\ref{euler-alpha-1r}$)--($\ref{euler-alpha-5r}$) in the sense of ($\ref{energy-formula-main-theorem}$).
\end{proof}

\section{  Discussions and comments}

As the solution to Euler-$\alpha$ equations  in Main Theorem satisfies the formula \eqref{energy-formula-main-theorem}, we find it is not necessary to require the regularity of initial data in $ \bm{H}^3(\Omega) \cap V$. It is a natural problem that one would like to look for the weak solution satisfying the formula \eqref{energy-formula-main-theorem} provided that the initial data are less regularity, for instance, $\bm{u}_0$ belongs to $\bm{H}^s(\Omega) \cap V$ for some positive number $s<3$. This is subtle to find such solution, because it is crucial not to obtain the estimates for the \textit{unfiltered} vorticity $q$ when initial data are less regular.

Further more, it is an interesting problem to hunt for the solutions to 3D Euler-$\alpha$ equations in exterior domain, which exist in some uniform time interval with respect to $\alpha.$ It is very difficult to do this, since  in 3D case, the equations for the \textit{unfiltered} vorticity $q$ are the transport equations given by
 $$\partial_t q + \bm{u} \cdot \nabla{q} = q\cdot\nabla\bm{u},$$
which is different from the formula \eqref{construction-equation-1}, and that one has no idea to  get the desired uniform estimates for hight order derivatives.  However, we can expect to obtain the limit problems of the 3D Euler-$\alpha$ equations past an obstacle  as both $\alpha$ and the diameter of the obstacle go to zero.%which is similar to the results for the viscous flows past obstacle in \cite{iftimie2003two}.
%At last, let us consider some questions naturally associated with the work we have presented. First, in the main theorem, we assume the initial data $\bm{u}_0$ belongs to $\bm{H}^3(\Omega) \cap V$. In future, we would like to weaken the regularity of $\bm{u}_0$ to get a weaker solution of Euler-$\alpha$ equations in exterior domain. Second, we plan to consider Euler-$\alpha$ equations in three dimension, focusing on well-posedness and vanishing $\alpha$ limit problems. Comparing with 2 dimension case, the global well-posedness problem in 3 dimension is more challenging, since the vorticity equation includes a deformmation term.
%%%%%%%%%%%%%%%%%%%%%%%%%%%%%%%%%%%%%%%%%%%%%%%%%%%%%%%%%%%%%%%%%%%%%%%%%%
%%%%%%%%%%%%%%%%%%%%%%%%%%%%%%%%%%%%%%%%%%%%%%%%%%%%%%%%%%%%%%%%%%%%%%%%%%
%%%%%%%%%%%%%%%%%%%%%%%%%%%%%%%%%%%%%%%%%%%%%%%%%%%%%%%%%%%%%%%%%%%%%%%%%%
\section*{Acknowledgement}
The authors would like to express their gratitude to Professor Guilong Gui for his discussion. The work of Aibin Zang was supported in part  by the National Natural Science Foundation of China (Grant no. 11771382, 12061080).
\normalem
\bibliographystyle{amsplain}
\bibliography{mybib}

%%%%%%%%%%%%%%%%%%%%%%%%%%%%%%%%%%%%%%%%%%%%%%%%%%%%%%%%%%%%%%%%%%%%%%%%%%%%%%%%%
%%%%%%%%%%%%%%%%%%%%%%%%%%%%%%%%%%%%%%%%%%%%%%%%%%%%%%%%%%%%%%%%%%%%%%%%%%%%%%%%%
%%%%%%%%%%%%%%%%%%%%%%%%%%%%%%%%%%%%%%%%%%%%%%%%%%%%%%%%%%%%%%%%%%%%%%%%%%%%%%%%%
\end{document}